\documentclass[11pt]{article}
\usepackage{amssymb, amsfonts, amsthm,amsmath}
\usepackage{setspace}

\usepackage{graphicx}
\usepackage{amsmath}
\usepackage{enumerate}
\usepackage{color}
\usepackage{verbatim}
\usepackage[margin=3cm]{geometry}

\usepackage{amsfonts,amssymb,amsopn,amscd}

\usepackage[colorlinks]{hyperref}
\usepackage[figure,table]{hypcap} 
\hypersetup{
   bookmarksnumbered,
   pdfstartview={FitH},
   citecolor={blue},
   linkcolor={red},
   urlcolor={black},
   pdfpagemode={UseOutlines}
}

\newcommand{\Ao}{\mathcal{A}_0}
\newcommand{\Af}{\mathcal{A}_f}
\newcommand{\eps}{\varepsilon}

\renewcommand{\P}{\mathbb{P}}
\newcommand{\E}{\mathbb{E}}

\newcommand{\R}{\mathbb{R}}

\newcommand{\D}{\mathcal{D}_R}
\newcommand{\V}{\mathrm{V}_N}
\newcommand{\G}{\mathcal{G}}
\newcommand{\K}{\mathcal{K}_{\mathbf{r}}(\mathcal{B}_0)}

\newcommand{\N}{\mathbb{N}}
\newcommand{\ve}{\varepsilon}
\renewcommand{\r}{\mathbf{r}}
\DeclareMathOperator{\bin}{Bin}

\theoremstyle{plain}
\newtheorem{Theorem}{Theorem}[section]
\newtheorem{Corollary}[Theorem]{Corollary}
\newtheorem{Lemma}[Theorem]{Lemma}
\newtheorem{Proposition}[Theorem]{Proposition}

\newtheorem{Claim}[Theorem]{Claim}

\theoremstyle{definition}

\newtheorem{Remark}[Theorem]{Remark}

\author{Elisabetta Candellero\footnote{Department of Statistics, University of Warwick, Coventry CV4 7AL, UK. Email: \texttt{elisabetta.candellero@gmail.com}} \and Nikolaos Fountoulakis\footnote{School of Mathematics, University of Birmingham, Birmingham B15 2TT, UK. Email: \texttt{n.fountoulakis@bham.ac.uk}, Research supported by a Marie Curie Career Integration Grant PCIG09-GA2011-293619 and an EPSRC Grant (No. EP/K019749/1).}}

\title{Bootstrap percolation and the geometry of complex networks\footnote{\textbf{Keywords}: Random geometric graph, hyperbolic plane, bootstrap percolation, percolation.}}

\begin{document}
\maketitle

\begin{abstract}
On a geometric model for complex networks (introduced by Krioukov et al.) we investigate the bootstrap percolation process. This model consists of random geometric graphs on the hyperbolic plane having $N$ vertices, a dependent version of the Chung-Lu model.
The process starts with infection rate $p=p(N)$. Each uninfected vertex with at least $\r\geq 1$ infected neighbors becomes infected, remaining so forever.
We identify a function $p_c(N)=o(1)$ such that a.a.s. when $p\gg p_c(N)$ the infection spreads to a positive fraction of vertices, whereas when $p\ll p_c(N)$ the process cannot evolve.
Moreover, this behavior is ``robust'' under random deletions of edges.
\end{abstract}


\section{Introduction}
\emph{Bootstrap percolation} is a deterministic process, characterized by a \emph{cascade behavior}, in which every vertex has two possible states: either \emph{infected} or \emph{uninfected} (sometimes also referred to as \emph{active} or \emph{inactive}, respectively).
A fixed integer $\r\geq 1$, called the \emph{activation threshold}, determines the evolution of the process, which occurs in rounds.

Initially, on the graph $G=G(V,E)$ there is a subset $\Ao \subseteq V$ which consists of infected vertices (vertices belonging to $V\setminus \Ao$ are uninfected) that can be selected deterministically or at random.

Subsequently, in each round, if an uninfected vertex has at least $\r$ infected neighbors, then it also becomes infected
and remains so forever. This is repeated until no more vertices become infected. We denote the final infected set by $\Af$.

This process was introduced by Chalupa, Leath and Reich~\cite{ChLeRe:79} in 1979 in the context of
magnetic disordered systems and has been re-discovered since then by several authors mainly due to its connections
with various physical models.

These processes have been used as models to describe several complex phenomena in diverse areas, from jamming 
transitions~\cite{tobifi06} and magnetic systems~\cite{sadhsh02} to neuronal activity~\cite{Am-nn, ACM11, ET09}. 
A short survey regarding applications of bootstrap percolation processes can be found in~\cite{AdL03}.

In the present paper, we consider a geometric framework for complex networks that was introduced by Krioukov et al.~\cite{ar:Krioukov}.
The theory of complex networks has been developed as a unifying mathematical framework that expresses features of 
a variety of networks: biological networks, large computer networks such as the Internet, the World Wide Web as well as
social networks that have been recently developed over these platforms. 
Experimental evidence (cf.~\cite{ChungLu2006},~\cite{ar:StatMechs}) has shown that these networks exhibit a few basic characteristics:
their degree distribution seems to follow a power-law, they exhibit local clustering and, finally, the typical distances between vertices
are small (this is known as the \emph{small world effect}).
In fact, most of these networks appear to have a degree distribution that has a power-law tail with exponent between 2 and 3 
(see~\cite{ar:StatMechs}). 

During the last 15 years there has been a continuous effort to develop models of
random networks which typically exhibit all the above features simultaneously. 
Among the most influential models was the Watts-Strogatz model of small 
worlds~\cite{ar:WatStrog98} and the Barab\'asi-Albert model~\cite{ar:BarAlb}, which is also known as the \emph{preferential attachment model}. 
The framework of Krioukov et al.~\cite{ar:Krioukov} represents the inherent inhomogeneity of a complex network with the use of the 
hyperbolic plane.
Intuitively, the intrinsic hierarchies 
that are present in a complex network induce a tree-like structure which is effectively embedded into the hyperbolic plane. 
The aim of this work is to shed some light on the evolution of a bootstrap percolation process and how this is determined by the geometry 
of the underlying network.

\subsection{Random geometric graphs on the hyperbolic plane and inhomogeneous random graphs}

The most common representations of the hyperbolic plane are the upper-half plane representation $\{z=x+iy \ : \ y > 0 \}$ as
well as the Poincar\'e unit disk which is simply the open disk of radius one, that is, $\{(u,v) \in \mathbb{R}^2 \ : \ 1-u^2-v^2 > 0 \}$.
Both spaces are equipped with the hyperbolic metric; in the former
case this is ${1\over y^2}dy^2$ whereas in the latter this is ${4}~{du^2 + dv^2\over (1-u^2-v^2)^2}$.
It is well-known that the (Gaussian) curvature in both cases is equal to $-1$ and that the two spaces are isometric. 
In fact, there are more representations of the hyperbolic plane of curvature $-1$, which are isometrically equivalent to the
above two. We will denote by $\mathbb{H}^2$ the class of these spaces.

In this paper, following the definitions in~\cite{ar:Krioukov}, we shall be using the \emph{native} representation of $\mathbb{H}^2$.
Under this representation, the ground space of $\mathbb{H}^2$ is $\mathbb{R}^2$ and every point $x \in \mathbb{R}^2$ whose
polar coordinates are $(r,\theta)$ has hyperbolic distance from the origin equal to its Euclidean distance, which is equal to $r$. 
Alternatively, the native representation can be thought of as a mapping of the Poincar\'e disk into $\mathbb{R}^2$. Under this
representation, a point $p$ in the Poincar\'e disk that is at hyperbolic distance $r$ from the origin and angle $\theta$ with respect to 
the horizontal axis is mapped to the point $p'$ of Euclidean distance $r$ from the origin of $\mathbb{R}^2$, preserving the angle. 

We are now ready to give the definitions of the two basic models introduced in~\cite{ar:Krioukov}.
Consider the native representation of the hyperbolic plane. Let $N$ be the number of vertices of the random graph. 
This is the parameter with respect to which we do asymptotics. 
For some fixed constant $\nu >0$, let $R>0$ satisfy $N= \lfloor \nu e^{R /2}\rfloor$.
For simplicity, we will omit $\lfloor \cdot\rfloor$ as this does not affect our calculations.  
We select randomly and independently $N$ points from the disk of radius $R$ centered at the origin $O$, which we denote by $\D$.

Each of these points is distributed as follows.  Assume that a random point $u$ has
\emph{polar} coordinates $(r, \theta)$. The angle $\theta$ is uniformly distributed in $(0,2\pi]$ and the probability density function of
$r$, which we denote by $\rho_N (r)$, is determined by a parameter $\alpha >0$ and is equal to
\begin{equation} \label{eq:pdf}
 \rho(r) = \rho_N (r) = \begin{cases}
\alpha {\sinh  \alpha r \over \cosh \alpha R - 1}, & \mbox{if $0\leq r \leq R$} \\
0, & \mbox{otherwise}
\end{cases}.
\end{equation}
The above distribution is simply the uniform distribution on $\D$, but \emph{on the hyperbolic plane of curvature $-\alpha^2$}.  
With elementary but tedious calculations, it can be shown that 
the length of a circle of radius $r$ (centered at the origin) on the hyperbolic plane of curvature $-\alpha^2$ is
$\frac{2\pi}{\alpha}~\sinh (\alpha r)$, whereas the area of the circle of radius $R$ (centered at the origin) is
$\frac{2\pi}{\alpha^2}(\cosh (\alpha R) - 1)$. 
Hence, when $\alpha = 1$, the above becomes the uniform distribution.

An alternative way to define this distribution is as follows. 
Consider the disk $\D'$ of radius $R$ around the origin $O'$ of (the native representation of) the hyperbolic plane of 
curvature $-\alpha^2$. Select $N$ points independently within $\D'$, uniformly at random. 
Subsequently, the selected points are projected onto $\D$ preserving their polar
coordinates. The projections of these points, which we will be denoting by $\V$, will be the vertex set of the random graph.

Note that the curvature in this case determines the rate of growth of the space. Hence, when $\alpha < 1$, the $N$ points 
are distributed on a disk (namely $\D'$) which has smaller area compared to $\D$. This naturally increases the density of those points 
that are located closer to the origin. Similarly, when $\alpha > 1$ the area of the disk $\D'$ is larger than that of $\D$, and most of the 
$N$ points are significantly more likely to be located near the boundary of $\D'$, due to the exponential growth of the volume. 

Given the set $\V$ on $\D$ we define the random graph $\G (N; \alpha, \nu)$  on $\V$, where  
two distinct vertices are joined precisely if they are within (hyperbolic) distance $R$ from each other.

\subsubsection{$\G (N; \alpha, \nu)$ and the Chung-Lu model} \label{sec:geom_asp}
The notion of \emph{inhomogeneous random graphs} was introduced by S\"oderberg~\cite{ar:s02} but
was defined more generally and studied in great detail by Bollob\'as, Janson and Riordan in~\cite{BJR}.
In its most general setting, there is an underlying compact metric space $\mathcal{S}$ equipped with a measure $\mu$ on its 
Borel $\sigma$-algebra. This is the space of \emph{types} of the vertices (defined below). A \emph{kernel} $\kappa$ is a bounded real-valued,
non-negative function on $\mathcal{S} \times \mathcal{S}$, which is symmetric and measurable. 
The vertices of the random graph can be understood as points in $\mathcal{S}$. 
If $x, y \in \mathcal{S}$, then the corresponding vertices are joined with probability 
${\kappa (x,y) \over N} \wedge 1$, independently of every other pair ($N$ is the total number of vertices). 
The points that are the vertices of the graph are approximately distributed according to $\mu$. 
More specifically, the empirical distribution function on the $N$ points converges weakly to $\mu$ as $N \rightarrow \infty$. 

Of particular interest is the case where the kernel function can be factorized and can be written $\kappa (x,y) = t(x)t(y)$; this is called a \emph{kernel of rank 1}. 
Here, the function $t(x)$ represents the weight of the type of vertex $x$ and, in fact, it is approximately its expected degree. 
The special case where $t(x)$ follows a distribution that has a power law tail was considered by Chung and Lu in a series of papers
~\cite{ChungLu1+}, \cite{ChungLuComp+} (see also~\cite{bk:vdH}). 

In the random graph $\G (N; \alpha, \nu)$ the probability that two vertices are adjacent has this form.  
The proof of this fact relies on Lemma~\ref{lem:relAngle}, which we will state and prove later. 
This provides an approximate characterization of what it means for two points $u,v$ to
have hyperbolic distance at most $R$ in terms of their \emph{relative angle}, which we denote by $\theta_{u,v}$. 
For this lemma, we need the notion of the \emph{type} of a vertex. 
For a vertex $v \in \V$, if $r_v$ is the distance of  $v$ from the origin, that is, the radius of $v$, then we set $t_v = R - r_v$ -- 
we call this quantity the \emph{type} of vertex $v$. 
As we shall shortly see, the type of a vertex is approximately exponentially distributed. If we substitute $R-t$ for $r$ in 
(\ref{eq:pdf}), then assuming that $t$ is fixed that expression becomes asymptotically equal to $\alpha e^{-\alpha t}$. 
By Lemma~\ref{lem:relAngle}, two vertices $u$ and $v$ of types $t_u$ and $t_v$ are within distance $R$ (essentially) if and only if 
$\theta_{u,v} < 2 \nu {e^{t_u/2} e^{t_v/2}}/N$. Hence, conditional on their types the probability that $u$ and $v$ are adjacent is 
proportional to ${e^{t_u/2} e^{t_v/2}}/N$. If we set $t(u)= e^{t_u/2}$, then $\P (t(u) \geq x) = \P (t_u \geq 2 \ln x)\asymp 
e^{-2\alpha \ln x} = 1/x^{2\alpha}$. In other words, the distribution of $t(u)$ has a power-law tail with parameter $2\alpha$. 
Thus, the random graph $\G (N; \alpha, \nu)$ is a \emph{dependent} version of the Chung-Lu model that emerges naturally from the 
hyperbolic geometry of the underlying space. The fact that this is a random geometric graph gives rise to the existence of local 
clustering, which is missing in the Chung-Lu model. There, most vertices have tree-like neighborhoods.  

In fact, it can be shown that the degree of a vertex $u$ in $\G (N; \alpha, \nu)$ that has type $t_u$ is approximately distributed as 
a Poisson random variable with parameter proportional to $e^{t_u/2}$. This is shown only implicitly in Lemma~\ref{lemma:deg_distrib}. 

Gugelmann, Panagiotou and Peter~\cite{ar:Kosta} showed that the degree of a vertex has a power law with exponent $2\alpha  +1$.
If $\alpha  > 1/2$, then the exponent of the power law may take any value greater than 2. 
When $1 > \alpha >  1/ 2$, this exponent is between 2 and 3.
They also showed that the average degree is a constant that depends on $\alpha$ and $\nu$,
and that the clustering coefficient (the probability of two vertices with a common
neighbor to be joined by an edge) of $\G (N;\alpha, \nu)$ is asymptotically  bounded away from $0$ with probability $1-o(1)$ as
$N\rightarrow \infty$. 

Furthermore, the second author together with Bode and M\"uller~\cite{bode_fountoulakis_mueller} showed that $\G (N;\alpha, \nu)$ with high probability has a \emph{giant component}, that is, a connected component containing a linear number of
vertices if $1 >\alpha > 1/2$. 
When $\alpha > 1$, the size of the largest component is bounded by a function that is \emph{sublinear} in $N$. 
Recently, Kiwi and Mitsche~\cite{ar:KiMits} showed that in the supercritical regime, 
the order of the second largest component is bounded by a polylogarithmic function of $N$ a.a.s.

\subsection{Results} 
The main result of this paper regards the size of the final set $\Af$ of a bootstrap percolation process with activation threshold $\r\geq 1$ 
on $\G (N; \alpha, \nu)$, with $1 > \alpha > 1/2$. 
We shall assume that the initially infected set $\Ao$ is a random subset of $\V$,
where each vertex is included independently with probability $p$. We call $p$ the \emph{initial infection rate}. 
We shall be assuming that $p$ \emph{does} depend on $N$. 
In fact, we will identify a critical infection rate (cf.\ Theorem~\ref{thm:main}) such that when $p$ ``crosses" this critical function
the evolution of the bootstrap process changes abruptly. This critical density converges to 0 as $N$ grows. 
Our results imply that a sub-linear initial infection results with high probability in the spread of the infection to 
a positive fraction of $\V$. 

Our hypothesis is that $1/2 < \alpha < 1$, whereby the random graph $\G (N;\alpha, \nu)$ exhibits power law degree distribution with 
exponent between 2 and 3. The second author and Amini~\cite{AmFou2014} showed a result analogous to Theorem~\ref{thm:main} 
for the Chung-Lu model with the exponent of the power law between 2 and 3.

In the present work, we additionally show that this phenomenon is robust under random deletions of the edges of 
$\G (N;\alpha, \nu)$. Assuming that we retain each edge independently with constant (independent of $N$) probability $\rho>0$, we let
$\G(N;\alpha, \nu,\rho)$ denote the resulting random graph. 
Since the number of edges of $\G(N;\alpha, \nu)$ is proportional to $N$ with high probability, it follows that if we allow 
$\rho = o(1)$, then $\G(N;\alpha, \nu,\rho)$ has only sub-linear components with high probability. 

\begin{Remark}[Notation]
We say that a sequence of events $\mathcal{E}_N$ on the space of graphs incurred by $\G (N;\alpha, \nu , \rho)$ occurs \emph{asymptotically almost surely (a.a.s.)} whenever $\mathbb{P} (\mathcal{E}_N) \to 1$ as $N \to \infty$. 
If $X_N$ is a random variable defined on $\G (N;\alpha, \nu , \rho)$, then we write $\liminf_{N \to \infty} X_N > 0$ \emph{a.a.s.}, 
if there exists a real number $c>0$ such that $X_N >c$, a.a.s.. 

For any two functions $f,g: \mathbb{N} \rightarrow \mathbb{R}^+$ we write $f(N) \gtrsim g(N)$ to denote that there is a constant $C>0$ such that $f(N) \geq C g(N)$ eventually as $N\rightarrow \infty$. 
Analogously we write $f(N) \lesssim g(N)$ if there is a constant $c>0$ such that $f(N) \leq c g(N)$ eventually as $N\rightarrow \infty$. 
Moreover, we write $f(N) \asymp g(N)$ if both $f(N) \lesssim g(N)$ and $g(N) \lesssim f(N)$ hold simultaneously. 
Finally, we write $ f(N) \ll g(N)$ or $f(N) \gg g(N)$ if $f(N)/g(N)\to 0 $ or $f(N)/g(N)\to \infty $ respectively.
\end{Remark}
Now we can state our main result.
\begin{Theorem} \label{thm:main} 
Let $\r\geq 2$ be an integer, let $\rho \in (0,1]$ and $1/2 <\alpha < 1$. Consider a bootstrap percolation process on $\G(N;\alpha, \nu,\rho)$ with activation threshold $\r$ and initial infection rate $p(N)$.
Then the following hold
\begin{itemize}
\item[(i)] If $\displaystyle p(N)N^{1/2\alpha}\to \infty$, then $\displaystyle \liminf_{N \rightarrow \infty} {|\Af| \over N} > 0$ a.a.s.; 
\item[(ii)] if $\displaystyle p(N)N^{1/2\alpha}\to \gamma \in \R^+$, then $\displaystyle \liminf_{N \rightarrow \infty} {|\Af| \over N} > 0$ with positive probability;
\item[(iii)] if $\displaystyle p(N)N^{1/2\alpha}\to 0$, then $\displaystyle |\Af|=|\mathcal{A}_0|$ a.a.s.
\end{itemize}
\end{Theorem}
The proof of Theorem \ref{thm:main} effectively makes use of a dense core that $\G (N; \alpha, \nu)$ has. 
%
Intuitively, the proof is based on considering the set of vertices appearing ``very close'' to the circumference of radius $R/2\alpha$.
In fact, if $p(N)$ is \emph{large enough}, then at least $\r$ of such vertices either belong to $\mathcal{A}_0$, or will be infected after
the first round (and hence they will spread the infection throughout the graph).
On the other hand, if $p(N)\ll N^{-\frac{1}{2\alpha}}$, then a.a.s. the process does not evolve and the final set of infected 
vertices will coincide with the initial set.

When considering the case for $\G(N;\alpha, \nu,\rho)$,
we have that this central core is a dense binomial random graph where each edge is present with probability $\rho$.
The condition on $p(N)$ ensures that the core becomes
completely infected even in this case. 
Thereafter, we apply an inductive argument which shows that with high probability the infection spreads from the core
to a positive fraction of $\V$. 

Note that when $\r=1$ the final set of infected vertices is the union of connected components that contain at least one infected vertex at the beginning of the process. 
Hence, a reformulation of part $(i)$ of Theorem \ref{thm:main} for $\r=1$ implies that the graph $\G(N;\alpha, \nu,\rho)$ for any $1/2 < \alpha < 1$ and any $\rho > 0$ contains a \emph{giant component}.
In other words, the graph $\G (N;\alpha,\nu)$ contains a giant component which is robust under random edge deletions. 
Let $\ell_1 (G)$ denote the number of vertices in a largest component of a graph $G$. 
\begin{Corollary} \label{cor:giant}
For all $\rho \in (0,1]$ we have a.a.s.
$$ \liminf_{N \rightarrow \infty}{\ell_1 (\G(N;\alpha, \nu,\rho)) \over N} > 0. $$
\end{Corollary}
A further consequence of part $(i)$ of Theorem~\ref{thm:main} is the existence of an $\r$-core in $\G(N;\alpha, \nu,\rho)$. Recall that 
for any integer $\r\geq 2$ the \emph{$\r$-core} of a graph $G$ is the maximum subgraph of minimum degree at least $\r$. This is a
well-studied notion in the theory of random graphs and hypergraphs 
(see for example~\cite{PitSpWorm},~\cite{Cooper2002},~\cite{Wormald}). 
Let $c_{\r}(G)$ denote the number of vertices of the $\r$-core of a graph $G$. 
\begin{Theorem} \label{cor:core}
For all integers $\r\geq 2$ and all $\rho \in (0,1]$ we have a.a.s.
$$\liminf_{N \rightarrow \infty}{c_{\r} (\G(N;\alpha, \nu,\rho)) \over N} > 0. $$
\end{Theorem}
In other words, Theorem~\ref{cor:core} implies that the $\r$-core of $G(N;\alpha,\nu)$ is robust under random edge deletions. 

The proof of Theorem~\ref{thm:main} is based on an inductive argument, and it is spread over Section \ref{sect:base_step} (the base step), and Section \ref{sect:inductive_step} (the inductive step).
Finally, in Section \ref{sect:proof_corollaries} we show the proofs of Corollary~\ref{cor:giant} and Theorem~\ref{cor:core}.

\paragraph{Acknowledgments.}
The authors are grateful to Peter M\"orters for suggesting the problem of the robustness of the giant component 
under random edge deletions.

\section{Preliminaries}
Throughout the paper, we will be working with the notion of the \emph{type} of a vertex, rather than its distance from the origin: denoting by $r_u$ the distance of vertex $u$ from the origin, its type is defined as $t_u:=R-r_u$.

It is not hard to show that the type of a vertex follows the exponential distribution with parameter $\alpha$.  
More specifically, it follows from (\ref{eq:pdf}) that for any $c<1$, uniformly over $t_u < c \cdot R$, we have 
\begin{equation}\label{eq:rho_approx}
\bar{\rho} (t_u) := \rho(r_u)= \alpha e^{-\alpha t_u}\bigl ( 1-o(1)\bigr ).
\end{equation}
We will use this asymptotic equality several times in our proofs, a proof of this easy fact can be found in~\cite{candellero_fountoulakis}. 
The above expression implies  that the probability that $t_u \geq R/(2 \alpha ) + \omega (N)$ is $o(1/N)$, provided that $\omega (N) \rightarrow \infty$ as $N \rightarrow \infty$. Therefore, a.a.s. all vertices have type that is bounded by $R/(2 \alpha ) + \omega (N)$,
where $\omega (N)$ can be any slowly growing function that tends to infinity.

\subsection{Distances on the hyperbolic plane}
We will need a general tool that will allow us to deal with distances on the hyperbolic plane (because these characterize 
whether or not two vertices are adjacent). 
The following lemma provides an \emph{almost characterization} for two points $u,v$ to have hyperbolic distance less than $R$ in terms of 
their types and their relative angles in $\D$. This is a key lemma whose proof is based on the hyperbolic law of cosines and allows 
us to estimate the probability that two vertices are adjacent. 
More specifically, assuming that the points have types $t_u$ and $t_v$, respectively, this condition is effectively 
an upper bound on the relative angle $\theta_{u,v}$ between $u$ and $v$ so that $d(u,v) < R$.  
\begin{Lemma}\label{lem:relAngle}
For any $\eps>0$ there exists an $N_0>0$ and a $c_0>0$ such that for any $N>N_0$ and $u,v\in\D$ with $t_u+t_v < R-c_0$
the following hold.
\begin{itemize}
			\item If $\theta_{u,v} < 2 (1-\eps )\exp\left(\frac{1}{2}(t_u+t_v-R)\right)$, then $d(u,v) < R$.
			\item If $\theta_{u,v}> 2(1+\eps )\exp\left(\frac{1}{2}(t_u+t_v- R)\right)$,
				then $d(u,v) > R$.
		\end{itemize}
\end{Lemma}
\begin{proof}
We begin with the hyperbolic law of cosines: 
\[
\cosh ( d(u,v))  = 
\cosh ( R- t_u) \cosh (R-t_v) - \sinh (R- t_u) \sinh (R-t_v) 
\cos ( \theta_{u,v} ).
\]
The right-hand side of the above becomes:
\begin{equation} \label{eq:coslaw} 
\begin{split} 
	&\cosh (R- t_u)\cosh (R-t_v) - \sinh (R- t_u ) \sinh (R-t_v ) \cos ( \theta_{u,v} ) \\
	& = {e^{ (2R- (t_u + t_v))} \over 4} \left[ \left(1+e^{-2 (R-t_u)}\right)\left(1+e^{-2 (R-t_v)}\right) \right .\\
	& \quad \left . -\left(1-e^{-2 (R-t_u)}\right)\left(1-e^{-2 (R-t_v)}\right) \cos (\theta_{u,v}) \right] \\
	& = {e^{ (2R- (t_u + t_v))} \over 4} \left[1- \cos (\theta_{u,v}) + \left(1+ \cos (\theta_{u,v}) \right) 
	\left( e^{-2 (R-t_u)} +  e^{-2 (R-t_v)}\right) \right . \\
	& \quad \left . + O\left( e^{-2 (2R- (t_u + t_v))}\right)\right].
\end{split}
\end{equation}
Therefore,
\begin{equation*}
	\begin{split}
    &\cosh ( d(u,v)) \leq \\
	&  {e^{ (2R- (t_u + t_v))} \over 4} \left[1- \cos (\theta_{u,v}) + 2 
	\left( e^{-2 (R-t_u)} +  e^{-2 (R-t_v)}\right)  + O\left( e^{-2 (2R- (t_u + t_v))}\right)\right].
	\end{split}
\end{equation*}
Since $t_u+t_v < R-c_0$, the last error term is $O(N^{-4})$. 
Also, it is a basic trigonometric identity that $1- \cos (\theta_{u,v}) = 2\sin^2 \left( {\theta_{u,v} \over 2} \right)$. 
The latter is at most ${\theta_{u,v}^2 \over 2}$. 
 Therefore, the upper bound on $\theta_{u,v}$ yields:
  \begin{equation*}  
\begin{split}
\cosh ( d(u,v)) &\leq {e^{ (2R- (t_u + t_v))} \over 4} \left( {\theta_{u,v}^2 \over 2} + 2 
	\left( e^{-2 (R-t_u)} +  e^{-2 (R-t_v)}\right) + O\left( {1\over N^4}\right)\right) \\
	&\leq {e^{ (2R- (t_u + t_v))} \over 4} \left( 2(1-\eps )^2 e^{ (t_u+t_v - R)} + 2 
	\left( e^{-2 (R-t_u)} +  e^{-2 (R-t_v)}\right) \right) + O\left( 1\right) \\
	&= (1-\eps)^2{e^{ R} \over 2} + {1\over 2}\left(e^{ ( t_u-t_v )} + e^{ (t_v-t_u)} \right) + O(1).
	\end{split}
 \end{equation*}  
At this point we choose $c_0$ such that $e^{-c_0}<\frac{\eps}{2}$, hence the above is bounded from above by
\[
(1-\eps)^2{e^{ R} \over 2} + \eps {1\over 2}\left(2e^{ ( t_u+t_v )}  \right) + O(1) < {e^{ R} \over 2},
\]
for $N$ large enough, since $t_u+t_v <  R-c_0$ and $t_u,t_v\geq0$.  
Also, since $\cosh ( d(u,v))> {1\over 2} e^{ d(u,v)}$, it follows that $d(u,v) < R$. 

To deduce the second part of the lemma, we consider a lower bound on (\ref{eq:coslaw}) using the lower bound on $\theta_{u,v}$: 
\begin{equation} \label{eq:coslawLow} 
    \begin{split} 
     \cosh ( d(u,v)) & \geq   {e^{ (2R- (t_u + t_v))} \over 4} \left(1- \cos (\theta_{u,v}) \right) +O(1) \\
     & \geq {e^{ (2R- (t_u + t_v))} \over 4} \left(1- \cos \left(2(1+\eps){e^{\frac{1}{2} (t_u+t_v- R)}}\right) \right)
+O(1). 
    \end{split}
\end{equation}
Using again that $1- \cos (\theta) = 2\sin^2 \left( {\theta  \over 2} \right)$ we deduce that 
\[
1- \cos \left( 2(1+\eps)e^{\frac{1}{2}(t_u+t_v - R)}\right)  = 2 \sin^2 \left( {1\over 2}~2(1+\eps)e^{\frac{1}{2}(t_u+t_v - R)}\right).
\]
Since $t_u+t_v < R - c_0$, it follows that $t_u+t_v-  R <- c_0$. So the latter is 
\[
\sin \left( (1+\eps)e^{\frac{1}{2}(t_u+t_v - R)}\right)\geq \frac{(1+\eps)e^{\frac{1}{2}(t_u+t_v - R)}}{2},
\]
for $N$ and $c_0$ large enough.
Substituting this bound into (\ref{eq:coslawLow}) we have 
\[
\cosh ( d(u,v))  \geq 
2\left (\frac{(1+\eps)e^{\frac{1}{2}(t_u+t_v - R)}}{2}\right )^2+O(1) 
= (1+\eps)^2\frac{e^{(t_u+t_v - R)}}{2}+O(1).
\]
Thus, if $d(u,v) \leq  R$, the left-hand side would be smaller than the right-hand side which would lead to a contradiction. 
\end{proof}

\subsection{Sketch of proof and the setup of the induction argument}\label{sect:C2_from_other_paper}
The proof of Theorem \ref{thm:main} relies on an inductive argument which shows
\begin{enumerate}
\item[1.] that if $p\gg N^{-1/(2\alpha)}$, then a.a.s.\ all vertices of type at least $R/2$ (which we call the \emph{core vertices})
become infected; 
\item[2.] how the infection spreads to the remaining vertices.
\end{enumerate}
To implement the second part, we divide the disk $\D$ into homocentric \emph{bands} and effectively show that \emph{if most of the vertices of a band are infected}, then this is also the case for the next band. This is the inductive step. 
The partition is defined as follows. We set $t_0:=R/2$ and for $i > 0$,
\begin{equation}\label{eq:def_ti}
t_i-2\ln\left(\frac{4\pi}{\nu(1-\varepsilon)^4} t_i\right)=\lambda t_{i-1},
\end{equation}
where $\lambda:=2\alpha-1$. Since $1/2 < \alpha < 1$, we have $0 < \lambda < 1$.
We set
\[
\mathcal{B}_0 := \{v\in\mathcal{D}_R \ : \ R/2<t_v\leq R\},
\]
and for $i>0$
\[
\mathcal{B}_i:= \{v \in \mathcal{D}_R \ : \ t_i\leq t_v < t_{i-1}\}.
\]
We shall restrict our analysis to $i < T$, where the value $T$ will be determined explicitly in Appendix~\ref{sect:step1}.
In particular, we have the following result whose proof is rather technical and hence deferred to Appendix \ref{sect:step1}.
\begin{Lemma}\label{lemma:T=O(lnR)}
The number of bands $T $ is of order $ O \left( \ln R \right)$.
\end{Lemma}
Also, for $i\geq 0$ denote by $\mathcal{C}_i$ the circle centered at the origin $O$ that has radius $R-t_i$. 
Hence $\mathcal{B}_i$ is delimited by $\mathcal{C}_i $ and $ \mathcal{C}_{i-1}$.

We will use some standard concentration inequalities to show that the number of vertices in each band is concentrated 
around its expected value. Once we have established this, we will condition on the sets of vertices that belong to each band. 

Let $\mathcal{N}_i$ denote the set of vertices in $\V$  that belong to $\mathcal{B}_i$, and let $N_i:=|\mathcal{N}_i|$.
\begin{Claim} 
We have 
\[
\mathbb{E} \left[ N_0 \right] = \Omega (N^{1- \alpha}).
\]
\end{Claim}
\begin{proof}
We use (\ref{eq:rho_approx}) and deduce that 
\begin{equation*}
\mathbb{E} \left[ N_0 \right] \geq (1-o(1)) N \alpha \int_{t_0}^{3R/4} e^{-\alpha t} dt 
= (1-o(1)) N \left( e^{-\alpha R/2} - e^{-3\alpha R /4} \right) = \Omega (N^{1-\alpha}). 
\end{equation*}
\end{proof}
The next claim deals with $i>0$.  
\begin{Claim} \label{claim:Ni}
Let $\eps > 0$. If $T$ is such that for $i < T$
\begin{equation}\label{eq:T_2}
e^{-\alpha (t_{i-1}-t_i)}<\ve,
\end{equation}
then for any $N$ sufficiently large and for every $0<i<T$ we have 
\[
(1 - \eps )^2 N e^{-\alpha t_i}\leq \mathbb{E} \left[ N_{i} \right] \leq (1 + \eps )^2 N e^{-\alpha t_i}.
\]
\end{Claim}
\begin{proof}
Note that by (\ref{eq:rho_approx}) (since $t_{i} \leq R/2$, for all $i>0$) we have uniformly for all $i >0$ 
\begin{equation} \label{eq:ExpNi} 
\mathbb{E} \left[ N_{i} \right] = (1-o(1)) N \alpha \int_{t_i}^{t_{i-1}} e^{-\alpha t} dt 
= (1-o(1)) N \left( e^{-\alpha t_i} - e^{-\alpha t_{i-1}} \right). \end{equation}
The upper bound follows trivially. 
For the lower bound we use (\ref{eq:T_2}) which together with (\ref{eq:ExpNi}) imply that for $N$ sufficiently large
\begin{equation} \label{eq:ExNiLow} 
\begin{split}  
\mathbb{E} \left[ N_{i} \right] = (1-o(1)) N e^{-\alpha t_i} \left( 1-  e^{-\alpha(t_{i-1}-t_i)}\right) 
\stackrel{(\ref{eq:T_2})}{>}  N e^{-\alpha t_i} (1-\ve)^2,
\end{split}
\end{equation}
which finishes the proof.
\end{proof}
For the sake of completeness, we recall here two classical results that will be used throughout the paper: \emph{Chernoff bounds} and the \emph{bounded-differences inequality}.

\paragraph{Chernoff bounds.}
Consider a binomial random variable $X\sim \bin (n,p)$. If $\mathbb{E} X = np \rightarrow \infty$, as $n \rightarrow \infty$,
then for every constant $\ell\geq 0$ we have
\begin{equation}\label{eq:chernoff}
\P(X\geq \ell) \geq 1-\exp{\left  ( 
-\frac{(\E X-\ell)^2}{2\E X} \right )}\geq 1-\exp{ \left ( -\frac{np}{4}\right )},
\end{equation}
provided that $n$ is large enough. 
Analogously, we have
\begin{equation}\label{eq:chernoff_up}
\P(X\geq 2\E X) \leq \exp{ \left ( -\frac{np}{8}\right )}.
\end{equation}
\paragraph{Bounded-differences inequality (Hoeffding-Azuma).}
Suppose $X_1,\ldots ,X_n$  is a collection of independent random variables, and $f:\mathbb{R}^n \rightarrow \mathbb{R}$ 
is a real-valued function. Let $c_1,c_2,\ldots,c_n$ satisfy 
\[
\sup_{x_1,x_2,\ldots,x_n,x_i'}\left | f(x_1,x_2,\ldots,x_i,\ldots,x_n)-f(x_1,x_2,\ldots,x_i',\ldots,x_n)\right |<c_i,
\]
for $1\leq i\leq n$. Now, let $ X:=f(X_1,X_2,\ldots,X_n)$. Then for every $0<\epsilon<1$ we have
\begin{equation}\label{eq:bdd_ineq}
\P\left (X<(1-\epsilon)\E X \right )\leq \exp{\left \{ -\frac{2\epsilon^2 (\E X)^2}{\sum_{i=1}^n c_i^2}\right \}}.
\end{equation}
Now we apply all these results to our setting, by showing the next lemma. 
\begin{Lemma}
For any $\varepsilon>0$ with high probability we have
\begin{equation}\label{eq:N_iFinalBounds}
 (1-\eps )^3 N e^{-\alpha t_i}\leq N_i \leq (1+\eps )^3 Ne^{-\alpha t_i}. 
\end{equation}
\end{Lemma}
\begin{proof}
Applying the Chernoff bound (\ref{eq:chernoff}), 
and since $T = O \left( \ln R \right)$ (by Lemma \ref{lemma:T=O(lnR)}), 
a simple first-moment argument shows that with probability 
$1- \exp \left(- \Omega \left( N^{1- \alpha} \right) \right)$ we have 
\[
(1- \ve) \E [ N_i ]\leq N_i\leq (1+ \ve) \E [ N_i ],
\]
for all $0< i < T$. 
More precisely, for $0 < i < T$, define the event
\begin{equation} \label{eq:N}  
\mathcal{E}_i:=\bigl \{(1- \ve) \E [ N_i ]\leq N_i\leq (1+ \ve) \E [ N_i ]\bigr \}.
\end{equation}
Since $\P(\mathcal{E}_i) \geq 1-e^{- \Omega \left( N^{1- \alpha} \right)}$ uniformly for every $i<T$, then the probability of the event 
\[
\mathcal{E}:=\bigcap_{i=1}^{T-1}\mathcal{E}_i
\]
is bounded from below by
\begin{equation}\label{eq:prob_E}
\P \left (\mathcal{E} \right ) \geq \prod_{i=1}^{T-1}\left ( 1-e^{- \Omega \left( N^{1- \alpha} \right)} \right )
\geq 1-\sum_{i=1}^{T-1}e^{- \Omega \left( N^{1- \alpha} \right)} \geq 1-Te^{- \Omega \left( N^{1- \alpha} \right)} .
\end{equation}
Again, using Lemma \ref{lemma:T=O(lnR)}, one has $\P \left (\mathcal{E} \right )=1-o(1) $.
Note that when $\mathcal{E}$ occurs, then for any $N$ sufficiently large, by Claim~\ref{claim:Ni} and (\ref{eq:N}) for all $i=1,\ldots , T$ we have the statement.
\end{proof}
\begin{Remark}
Throughout the rest of the paper, we will be assuming that the event $\mathcal{E}$ has been realized, and hence that (\ref{eq:N_iFinalBounds}) holds.
\end{Remark}
Now we show an intermediate result about the degree distribution of a vertex $v\in \V$ conditional on its type.
For two random variables $X$ and $Y$ we write $X \preccurlyeq Y$ to denote that the random variable $Y$ \emph{stochastically dominates}
$X$.
\begin{Lemma}\label{lemma:deg_distrib}
Let $\mathbf{d}_v$ denote the degree of a vertex $v$ conditional on its type $t_v$.
There are two constants $0<H_1<H_2$ such that for any $N$ sufficiently large, if $t_v<R/2\alpha+\omega(N)$, then
we have
\[
\bin\left(N,H_1 \frac{e^{t_v/2}}{N}\right )\preccurlyeq \mathbf{d}_v \preccurlyeq \bin\left(N,H_2 \frac{e^{t_v/2}}{N}\right ).
\]
\end{Lemma}
\begin{proof}
For every vertex $u\in \V$, the probability that $u\sim v$ (conditional on $t_v$) can be bounded from above using Lemma \ref{lem:relAngle} and (\ref{eq:rho_approx}) as follows
\[
\P(u\sim v\mid t_v)\lesssim  \int_0^{R-t_v-C}\frac{e^{(t_u+t_v)/2}}{N}\bar{\rho}(t_u) dt_u + \P(t_u>R-t_v-C \ | \ t_v),
\]
where $C>0$ is an arbitrary large constant.
On the other hand, a lower bound is given by 
\[
\P(u\sim v\mid t_v)\gtrsim  \int_0^{R-t_v-C}\frac{e^{(t_u+t_v)/2}}{N}\bar{\rho} (t_u) dt_u.
\]
The first step of the proof consists in estimating the integral
\[ 
\mathbf{I_1}:=\int_0^{R-t_v-C}\frac{e^{(t_u+t_v)/2}}{N}\overline{\rho} (t_u)dt_u,
\] 
and subsequently in showing that the quantity 
\[
\mathbf{I_2}:=\P(t_u>R-t_v-C \ | \ t_v)
\]
satisfies
\begin{equation}\label{eq:I2=o(I1)}
\mathbf{I_2}=o(\mathbf{I_1}), 
\end{equation}
uniformly over $t_v$. 
We start with the first task, where we use (\ref{eq:rho_approx}) and write 
\[
\begin{split}
\mathbf{I_1} 
& = (1+o(1))\int_0^{R-t_v-C}\frac{e^{(t_u+t_v)/2}}{N}e^{-\alpha t_u}dt_u \\
& = (1+o(1))\frac{e^{t_v/2}}{N}\int_0^{R-t_v-C}e^{-(\alpha-1/2)t_u}dt_u 
\stackrel{\alpha>1/2}{\asymp} \frac{e^{t_v/2}}{N}.
\end{split}
\]
We can bound  $\mathbf{I_2}$ similarly:
\[
\begin{split}
\mathbf{I_2}
& \lesssim \int_{R-t_v-C}^R e^{-\alpha t_u} dt_u 
 \asymp e^{-\alpha(R-t_v-C)} \asymp e^{-\alpha(R-t_v)}e^{\alpha C} \asymp \left( \frac{e^{t_v/2}}{N} \right)^{2\alpha}.
\end{split}
\]
By direct comparison we see that (\ref{eq:I2=o(I1)}) holds, since $\alpha>1/2$.

It is now easy to see that the above must imply that there are two constants $0<H_1<H_2$ such that 
for any $N$ sufficiently large
\[
H_1\frac{e^{t_v/2}}{N} \leq \P(u\sim v\mid t_v) \leq H_2\frac{e^{t_v/2}}{N}.
\]
Therefore, since the (conditional) degree of $v$ is given by
\[
\bin \left (N,\P(u\sim v\mid t_v)\right ),
\]
for any $u\in \V$, then we have the statement.
\end{proof}
\begin{Remark}\label{remark:K_2>1}
Without loss of generality, we can choose the constant $H_2$ such that $H_2>1$.
\end{Remark}

Of particular interest will be what we call the \emph{light} degree of a vertex $v$. 
This is defined as follows: given a constant $C>0$ (that is to be specified later during our proof, cf.\ Appendix \ref{sect:step1}), the \emph{light degree} of vertex $v$ is the number of neighbors of $v$ that have type less than $C>0$.
To emphasize the dependence on $C$ we denote it 
by $\mathbf{d}_C (v)$. Arguing as in the proof of the above lemma we can show the following statement.  
\begin{Lemma}\label{lemma:deg_distrib_light}
Let $C>0$ and let $\mathbf{d}_C(v)$ be the light degree of a vertex $v$ conditional on its type $t_v$.
There are two constants $0<H_1'<H_2'$ that depend on $\alpha$ and $C$ such that for any $N$ sufficiently large, if
$t_v<R/2\alpha+\omega(N)$, then
we have
\[
\bin\left(N,H_1' \frac{e^{t_v/2}}{N}\right )\preccurlyeq \mathbf{d}_C (v) \preccurlyeq \bin\left(N,H_2' \frac{e^{t_v/2}}{N}\right ).
\]
\end{Lemma}

\section{Base Step for Induction}\label{sect:base_step}
In this section we show that under assumption $(i)$ (resp.\ $(ii)$) of Theorem~\ref{thm:main}, with high probability (resp.\ positive probability) all vertices in $\mathcal{B}_0$ become infected. This will be the base case for an inductive argument that will eventually show how the 
infection will spread from $\mathcal{B}_0$ to almost every vertex in $\mathcal{B}_i$, for all $i< T$. Finally, we will show that the 
number of these vertices is linear a.a.s. 


\subsection{The bootstrap percolation process inside $\mathcal{B}_0$}
Let us condition on having $N_0$ vertices in $\mathcal{B}_0$, where $N_0$ satisfies (\ref{eq:N_iFinalBounds}). 
Note that by the triangle inequality, any two of them are within distance at most $R$ and, therefore, they form a clique.

The random deletion of the edges incurs a random graph on $N_0$ vertices, whose edges appear independently with probability $\rho$. 
That is, this subgraph of $\G(N;\alpha,\nu, \rho)$ is distributed as the binomial random graph $G(N_0,\rho)$. 
Since $\rho$ is constant, $G(N_0,\rho)$ is connected a.a.s., and it will play the role of the ``seed" graph from which the bootstrap 
process evolves.

Now assume that we infect each vertex independently with probability $p=p(N)$, where $p N^{1/2\alpha} =\Omega(1)$ (as in 
Theorem~\ref{thm:main}$(i)$--$(ii)$). Let $\G_C$ denote the subgraph of $\G (N; \alpha, \nu , \rho)$ which is induced by the vertices of type less than $C$ together with the vertices of $\mathcal{B}_0$. The constant $C$ will be specified later in our proof. 
We will show that if $p(N)$ satisfies the conditions of Theorem~\ref{thm:main} \emph{(i)} and \emph{(ii)},  
then $\mathcal{B}_0$ becomes infected a.a.s., in the former case or with positive probability, in the latter case, through the subgraph
$\G_C$. More precisely, having shown this, we will further show that the infection spreads to most of the vertices of type at least $C$. 
Clearly, the two stages use disjoint sets of edges and we can expose them separately. 

We now proceed with the analysis of the first stage. 
By Lemma~\ref{lemma:deg_distrib_light}, conditional on the type $t_v$, a vertex $v$ has expected \emph{light} degree (before the random 
edge deletions) 
$ H_1' e^{t_v/2}\leq \E \mathbf{d}_C (v) \leq  H_2' e^{t_v/2}$, where $H_1', H_2'$ depend only on $\alpha$ and $C$. 
Thus, each vertex $v$ in the set
\begin{equation}\label{eq:tv}
\mathfrak{D}_1:=\left \{ u \ : \ \frac{R}{2\alpha}-\omega(N) \leq t_u \leq \frac{R}{2\alpha}+\omega(N)\right \},
\end{equation}
has (conditional) expected light degree
\[
\E \mathbf{d}_C (v) \geq H_1' \exp( t_v/2). 
\]
Hence, applying (\ref{eq:chernoff}) we obtain that 
\[
\P\left ( \mathbf{d}_C (v) \leq H_1'\frac{e^{t_v/2}}{2}\mid t_v\right )\leq e^{-\Theta \left ( e^{t_v/2} \right)},
\]
which tends to zero exponentially fast for $v \in \mathfrak{D}_1$, implying that in fact the conditional light degree of a vertex 
$v\in \mathfrak{D}_1$ is, with high probability, 
\[ \mathbf{d}_C (v) \geq H_1' e^{t_v/2}/2 \geq \frac{H_1'}{2}~\exp\left( \frac{1}{2}\left( \frac{R}{2\alpha}-\omega(N) \right) \right) 
\asymp N^{\frac{1}{2\alpha}} e^{- \omega(N)/2}. \]

Now assume we are in case $(i)$ of Theorem~\ref{thm:main}, and set $\varphi(N):= p(N)N^{1/2\alpha}$; hence $ \varphi(N)$ is some increasing function growing to infinity.
This implies that for any vertex in $\mathfrak{D}_1$, the (conditional) expected number of its neighbors that have type at most $C$ and are externally infected is at least (up to multiplicative constants) 
\[
p(N) \rho N^{\frac{1}{2\alpha}} e^{- \omega(N)/2} \asymp \varphi(N)e^{-\omega(N)/2}.
\]
Choosing $\omega(N)$ such that $ e^{\omega(N)}=o(\varphi(N))$ one has 
\[
\varphi(N) e^{-\omega(N)/2}\to \infty,
\]
which further implies that as $N\to \infty$ we have
\[
\P\bigl (v\in \mathfrak{D}_1\textnormal{ has at least }\r\textnormal{ initially infected neighbors}\bigr )\to 1.
\]
Also, it is not hard to see that a.a.s.\ $\mathfrak{D}_1$ contains at least $\r$ vertices. Let $v_1,\ldots, v_{\r}$ be an 
arbitrary collection of those.  Conditional on their degrees being as above, the FKG inequality implies that as $N \rightarrow \infty$ 
$$\P \bigl ( v_1,\ldots, v_{\r} \textnormal{ become infected} \bigr) \rightarrow 1.$$ 
(The FKG inequality is applied on the product space of initial infections, using the fact that the event that a vertex has at least 
$\r$ initially infected neighbors is increasing.)
In other words, the random graph $G(N_0,\rho)$ contains at least $\r$ vertices which become infected after round 1. 
Now, Theorem 5.8 in~\cite{janson_luczak_turova_vallier} implies that a.a.s. the bootstrap percolation process in $G(N_0,\rho)$ with 
this initial set of infected vertices results in complete infection of the vertices in $\mathcal{B}_0$.  

Similarly, now assume we are in case $(ii)$ and set $ \lim_{N\to \infty} \varphi (N)=\gamma >0$.
It is not hard to show that the set  
\begin{equation}\label{eq:tv2}
\mathfrak{D}_2:=\left \{ v \ : \ \frac{R}{2\alpha } \leq t_v\leq \frac{R}{2\alpha}+\omega(N)\right \},
\end{equation}
is non-empty with probability which remains bounded away from 0 as $N$ grows. In fact, the number of vertices therein follows asymptotically a Poisson distribution with 
constant parameter. 
In particular, looking at what happens asymptotically when $N\to \infty$, we have that with probability bounded away from $0$, it contains at least $\r$ vertices.
Conditional on their degree, the expected number of infected neighbors that have type less than $C$ of anyone of these vertices is
\[ \rho \varphi(N)\rightarrow \rho \gamma.
\]
Hence, as $N\to \infty$ the probability that a vertex 
$v\in \mathfrak{D}_2$ has at least $\r$ initially infected neighbors of type less than $C$ is bounded from below by some
positive constant.
Using again the FKG inequality, we deduce that with asymptotically positive probability these vertices become infected after 
the first round, and a.a.s. subsequently infect all vertices in $\mathcal{B}_0$. 

Therefore, from now in cases $(i)$ and $(ii)$ we will use these vertices as the \emph{root} of the infection for those vertices that 
have type at least $C$. 

We will denote by $\mathcal{K}_{\r}(\mathcal{B}_0)$ the connected
component induced by the infected vertices containing the (infected) vertices in $\mathcal{B}_0$. 
In our construction, we will inductively ``discover" a subgraph of $\mathcal{K}_{\r}(\mathcal{B}_0)$.

In particular, assume that we have discovered a certain set of vertices in $\mathcal{K}_{\r}(\mathcal{B}_0)$ contained in the union 
$\bigcup_{j=0}^{i-1}\mathcal{B}_j$. Our inductive step consists in proving that with high probability there are many vertices in
$\mathcal{B}_i$ which are connected to those in $\mathcal{K}_{\r}(\mathcal{B}_0)$ through at least $\r$ edges that are \emph{retained after the percolation process}.

\subsection{The base step starting from  $\mathcal{B}_0$}\label{sect:black_block_1}
In this section we show that the infection spreads from $\mathcal{B}_0$ to the external bands by showing that the core is \emph{very well connected} with vertices in the outer bands.
More precisely, we show that with very high probability 
%
outer vertices are contained in the disk of radius $R$ of at least $\r$ vertices which are located inside $\mathcal{B}_0$, 
and, moreover, at least $\r$ of such connections survive the percolation process.
Now we proceed with making this approach rigorous.

For $i\geq 1$, define
\begin{equation}\label{eq:def_theta^i} 
\theta^{(i)}:=2(1-\varepsilon ) e^{\frac{1}{2}(t_i+t_{i-1}-R)}. 
\end{equation}
Consider a point $u$ on $\mathcal{B}_{i-1}$. Lemma \ref{lem:relAngle} implies that all points $v \in \mathcal{B}_i$ with $\theta_{u,v} 
\leq \theta^{(i)}$ belong to the disk of radius $R$ that is centered at $u$. The set of these points is illustrated by the shaded area 
in Figure~\ref{fig:shadow}.

\begin{figure}[h!] \label{fig:shadow}
\begin{center}
\includegraphics[scale=0.7]{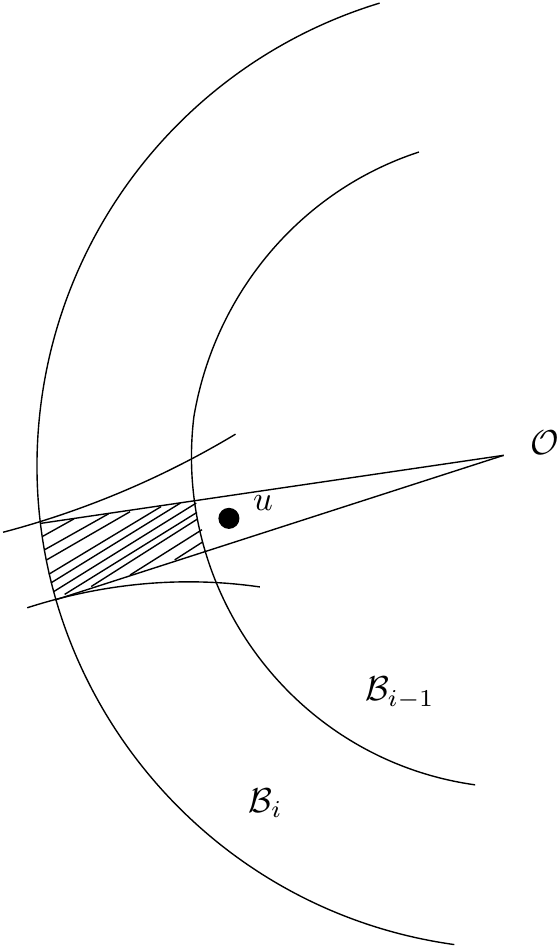}
\caption{Set of points $v \in \mathcal{B}_i$ with $\theta_{u,v} \leq \theta^{(i)}$: by Lemma~\ref{lem:relAngle}, they are all inside the disk of radius $R$ centered at $u$.}
\end{center}
\end{figure}

For every $i > 0$ we set
\begin{equation}\label{eq:def_Bi}
B_i:=\frac{\theta^{(i)}}{t_i}.
\end{equation}
Now consider the circle $ \mathcal{C}_1$ (i.e., the set of points of type $t_1$), divide it into consecutive \emph{blocks} of angle $B_1$,
and discard the (unique, if it exists) remaining block of angle smaller than $B_1$.

In what follows we shall be referring to these blocks as
\[
I_1^{(1)}, I_2^{(1)}, \ldots, I_{K_1}^{(1)},
\]
where the subscript is the index of the block (with $K_1:=\lfloor 2\pi/B_1\rfloor $) and the superscript denotes the index of the circle. 

For each $ j\in \{1, \ldots , K_1\}$, corresponding to each block $I_j^{(1)}$ we define a region (which will be called \emph{active area} -- 
cf.\ Figure~\ref{fig:active_areas}) $\mathcal{A}_j^{(1)} $ as follows:
\[
\mathcal{A}_j^{(1)}:= \{ x=(r,\theta) \ : \ R-t_1<r<R-t_0, \ \theta \in I_j^{(1)}\}.
\]

\begin{figure}[h!] \label{fig:active_areas}
\begin{center}
\includegraphics[scale=0.5]{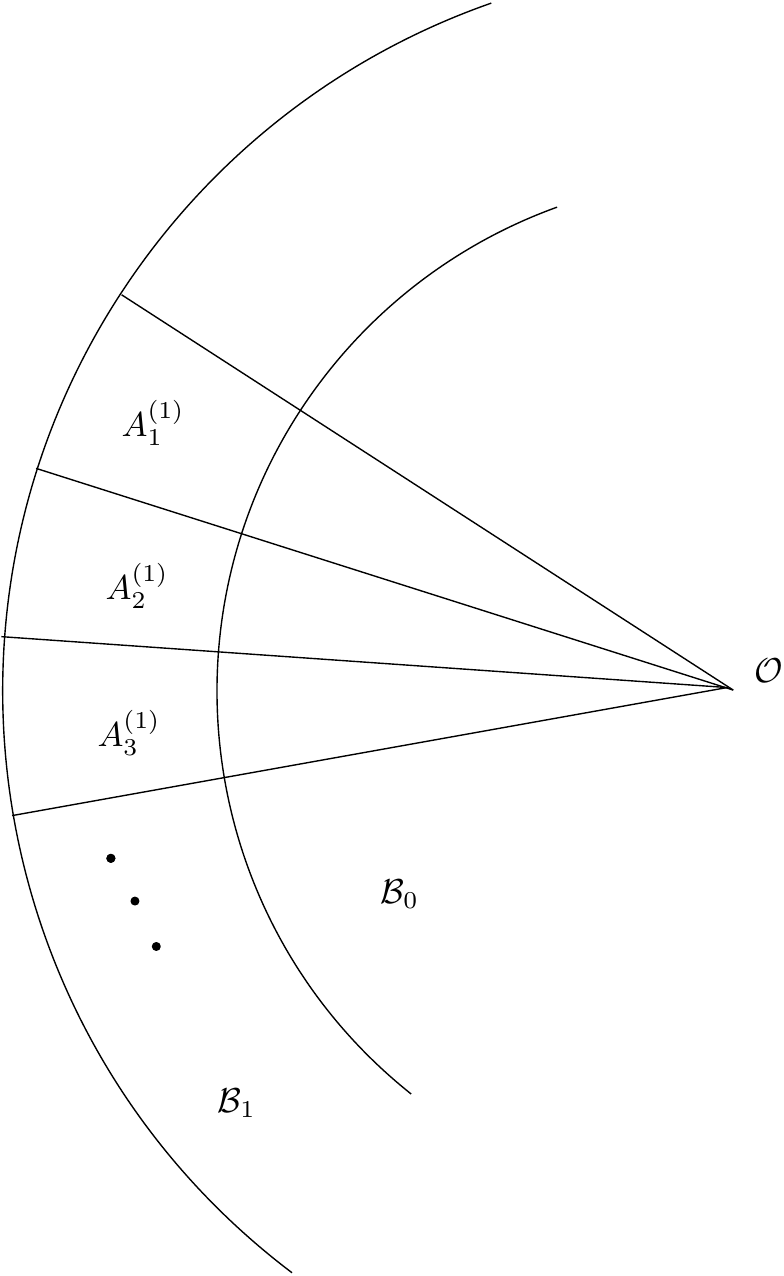}
\caption{Active areas in $\mathcal{B}_1$.}
\end{center}
\end{figure}

By taking the union over all blocks we define
\[
\mathcal{A}^{(1)}:=\bigcup_{j=1}^{K_1}\mathcal{A}_j^{(1)},
\]
with $r_i=R-t_i$, the above can be expressed as
\begin{equation}\label{eq:def_active_area_1}
\mathcal{A}^{(1)} :=\left \{x=(r,\theta) \ : \ x \in \mathcal{A}_j^{(1)}, \textnormal{ for some }j\in \{1, \ldots , K_1\}\right \}.
\end{equation}
We color each block $I_j^{(1)}$ according to the following rule:
\begin{itemize}
\item[(i)] \emph{black}, if for some integer $k\geq \r$, there are vertices $x_1^{(0)},x_2^{(0)},\ldots, x_k^{(0)}$ in $\mathcal{B}_0$ such that
$I_j^{(1)}$ is completely contained in the $k$ disks of radius $R$ centered at $x_1^{(0)},x_2^{(0)},\ldots, x_k^{(0)}$;
\item[(ii)] \emph{white}, otherwise.
\end{itemize}
If a block $I_j^{(1)}$ is black, then any vertex that falls inside $\mathcal{A}_j^{(1)}$ will be connected 
to $x_1^{(0)},x_2^{(0)},\ldots, x_k^{(0)}$.
We show the following intermediate result.
\begin{Lemma}\label{lemma:P_black_block_1}
The expected number of black blocks is bounded from below by
\[
K_1(1-e^{-t_1}).
\]
\end{Lemma}
\begin{proof}
If a vertex $v \in \mathcal{B}_1$ falls inside the active area $\mathcal{A}_j^{(1)}$ of a black block $I_j^{(1)}$, then $v$ will be
connected to at least $\r$ vertices in $\mathcal{B}_0$.
As a consequence, if all vertices in $\mathcal{B}_0$ are infected, then vertex $v$ will become infected as well.
Moreover, if at least $\r$ of these edges incident to $v$ are retained during the percolation process, then vertex $v$ will be added to $\mathcal{K}_{\r}(\mathcal{B}_0)$.

Define the event
\[
\mathbf{B}_j^{(1)}:=\{\textnormal{block }I_j^{(1)} \textnormal{ is black}\}.
\]
The probability of $\mathbf{B}_j^{(1)}$ is the probability of $I_j^{(1)}$ to be completely contained in the disk of $k\geq \r$ vertices in $\mathcal{B}_0$. 
The probability that the disk of a certain vertex in $\mathcal{B}_0$ contains $I_j^{(1)}$ is at least $\theta^{(1)}(1-t_1^{-1}) / (2\pi)$.
Note that these events are independent. 
Therefore
\begin{equation} \label{eq:block_prob}
\begin{split}
\P (\mathbf{B}_j^{(1)}) 
& \geq  \P \left (\bin \left ( N_0, \frac{\theta^{(1)}(1-t_1^{-1})}{2\pi}\right )\geq \r\right ) \geq 1-e^{-N_0\theta^{(1)}/(4\pi)},
\end{split}
\end{equation}
where the second inequality would follow from (\ref{eq:chernoff}), provided we had $N_0 \theta^{(1)} \rightarrow \infty$.
Now, to bound this expression for arbitrary $i$, we use (\ref{eq:N_iFinalBounds}). In particular, we have 
\[
2\nu(1-\varepsilon)^4 e^{{1 \over 2}(t_i + t_{i-1}) - \alpha t_{i-1}}\leq \theta^{(i)} N_{i-1} \leq 2\nu(1+\varepsilon)^3 e^{{1 \over 2}(t_i + t_{i-1}) - \alpha t_{i-1}}.
\]
Furthermore, from (\ref{eq:def_ti}) it follows
\[
{1 \over 2}(t_i + t_{i-1}) - \alpha t_{i-1} = \ln \left({4\pi \over 2\nu (1-\varepsilon)^4} t_i \right),
\]
which in turn yields
\begin{equation}\label{eq:theta^ixN_i-1}
4\pi t_i\leq \theta^{(i)} N_{i-1} \leq \frac{(1+\ve)^3 4\pi}{(1-\ve)^4}t_i.
\end{equation}
Thus $N_0 \theta^{(1)} \rightarrow \infty$. 
As $\theta^{(1)} N_{0} \geq 4 \pi t_1$, now (\ref{eq:block_prob}) implies
\[ 
\P (\mathbf{B}_j^{(1)})  \geq 1-e^{-t_1}.
\]
Denote by $\mathcal{S}_1 $ the set of all black blocks and by $S_1$ its cardinality. 
By what argued above, we have that $ \E S_1$ is bounded from below by the quantity $L_1$, defined as 
\begin{equation}\label{eq:def_L1}
L_1 := K_1 \left (1-e^{-t_1}\right ),
\end{equation}
which finishes the proof.
\end{proof}
Let $\mathcal{S}_{i}$ be the collection of black blocks that will similarly be defined on $\mathcal{C}_{i}$, for $i\geq 1$. 
From now on, let $\Theta_{i}$ denote the total angle that these blocks cover, and set $\Theta_0 = 2 \pi$. 

We will show that a sufficient number of them are black, so that $\Theta_i > \pi/2$ a.a.s.\ (see Appendix~\ref{sect:step1}).
Below we show this for $S_1$, starting with the following Lemma. 
%
%
\begin{Lemma}\label{lemma:concentration_S_1}
There is a decreasing function $\varepsilon_1:=\varepsilon_1(N)=o(1)$ such that with high probability we have $S_1>L_1(1-\varepsilon_1)$.
\end{Lemma}
\begin{proof}
Note that changing the position of one vertex in $\mathcal{B}_0 $ affects the number of black blocks on $ \mathcal{C}_1$ by at most 
\[
2 \frac{ \theta^{(1)}}{B_1}=2\frac{\theta^{(1)}}{\theta^{(1)}/t_1}=2t_1.
\]
This is the case as the disk of radius $R$ around each vertex in $\mathcal{B}_0$ contains at most $\theta^{(1)}/B_1$ intervals $I_j^{(1)}$.
Hence, by inequality (\ref{eq:bdd_ineq}) we have
\begin{equation}\label{eq:azuma_ve_1}
\begin{split}
\P & \bigl [ S_1<(1-\varepsilon_1)\E S_1 \mid N_0, \Theta_0 \bigr ] 
\leq \exp{\left \{ -\frac{\varepsilon_1^2 \E^2 S_1 }{N_0 (2t_1)^2}\right \} } \leq 
\exp{\left \{ -\frac{\varepsilon_1^2 L_1^2 }{N_0 (2t_1)^2}\right \} }.
\end{split}
\end{equation}
Now we need to find $\varepsilon_1$ such that the absolute value of the exponent above tends to infinity. Using the definitions of $L_1(\asymp 1/B_1)$  and  of $B_1(= \theta^{(1)}/t_1)$ (cf.\ (\ref{eq:def_L1}) and (\ref{eq:def_Bi}) respectively), this reduces to finding $\varepsilon_1$ such that
\[
M_1^{(2)}(N):=\frac{\varepsilon_1^2}{B_1^2 N_0 t_1^2}\asymp
\frac{\varepsilon_1^2}{(\theta^{(1)})^2 N_0}\to \infty .
\]
Here we recall relation (\ref{eq:theta^ixN_i-1}) (that is, $\theta^{(1)} N_0 \asymp t_1$), which implies that the above holds if
\begin{equation}\label{eq:ve_1}
\varepsilon_1\gg \sqrt{\theta^{(1)}t_1}.
\end{equation}
Now, by (\ref{eq:azuma_ve_1}) we have that the \emph{number} of black blocks $S_1 $ (at this first stage) is, with high probability, bounded from below by
\begin{equation}\label{eq:def_L_1}
L_1':= L_1(1-\varepsilon_1),
\end{equation}
for any function $\ve_1=\ve_1(N)=o(1)$ for which (\ref{eq:ve_1}) holds.
\end{proof}
Now, let $\Theta_1$ denote the total angle covered by the blocks in $\mathcal{S}_1$.
We conclude this section by showing the following result. 
\begin{Proposition}\label{prop:Theta_1>pi}
Asymptotically almost surely we have
\[
\Theta_1>\pi.
\]
\end{Proposition}
\begin{proof}
By definition, $\Theta_1 \geq B_1 L_1'$ which, by 
Lemmas \ref{lemma:P_black_block_1} and \ref{lemma:concentration_S_1} is bounded from below by
\begin{equation}\label{eq:Theta_1>pi}
\Theta_1\geq 2\pi (1-e^{-t_1})(1- \eps_1 ) > \pi,
\end{equation}
provided that $N$ is sufficiently large.
\end{proof}

\section{Inductive Step}\label{sect:inductive_step}
In this section we show that the infection spreads throughout the graph, and eventually reaches a linear fraction of vertices.

\subsection{Black Blocks}\label{sect:black_blocks_inductive}
Assume that at step $i-1$ (for any $2\leq i< T-1$) we have a collection of pairwise disjoint black blocks $\mathcal{S}_{i-1}$ each covering an angle
equal to $B_{i-1}$ (recall its definition from (\ref{eq:def_Bi})). The total angle that is covered by these black blocks is equal to
$\Theta_{i-1}= S_{i-1} B_{i-1}$, which we will show to be at least $\pi/2$ a.a.s. (see Appendix~\ref{sect:step1}).

We need to show that when passing from level $i-1$ to level $i$, we still obtain a sufficient number of black blocks. We proceed with the details.
First, we consider the projection of the blocks in $ \mathcal{S}_{i-1}$ (i.e., the set of black blocks found on $\mathcal{C}_{i-1}$) onto the
outer boundary of $\mathcal{B}_i$ (i.e.,\ on the circle $\mathcal{C}_i $), and declare the images of these projections \emph{uncolored}. 

Subsequently, we divide each block (in the projection) into three parts, namely:
\begin{itemize}
\item Parts (1) and (3): the first and the last part of the block, both of angle $\theta^{(i)}$;
\item Part (2): the remaining (central) part of the block, which has angle $B_{i-1}-2\theta^{(i)} $.
\end{itemize}
Such a subdivision is shown in Figure~\ref{fig:black_block}.

\begin{figure}[h!] \label{fig:black_block}
\begin{center}
\includegraphics[scale=0.5]{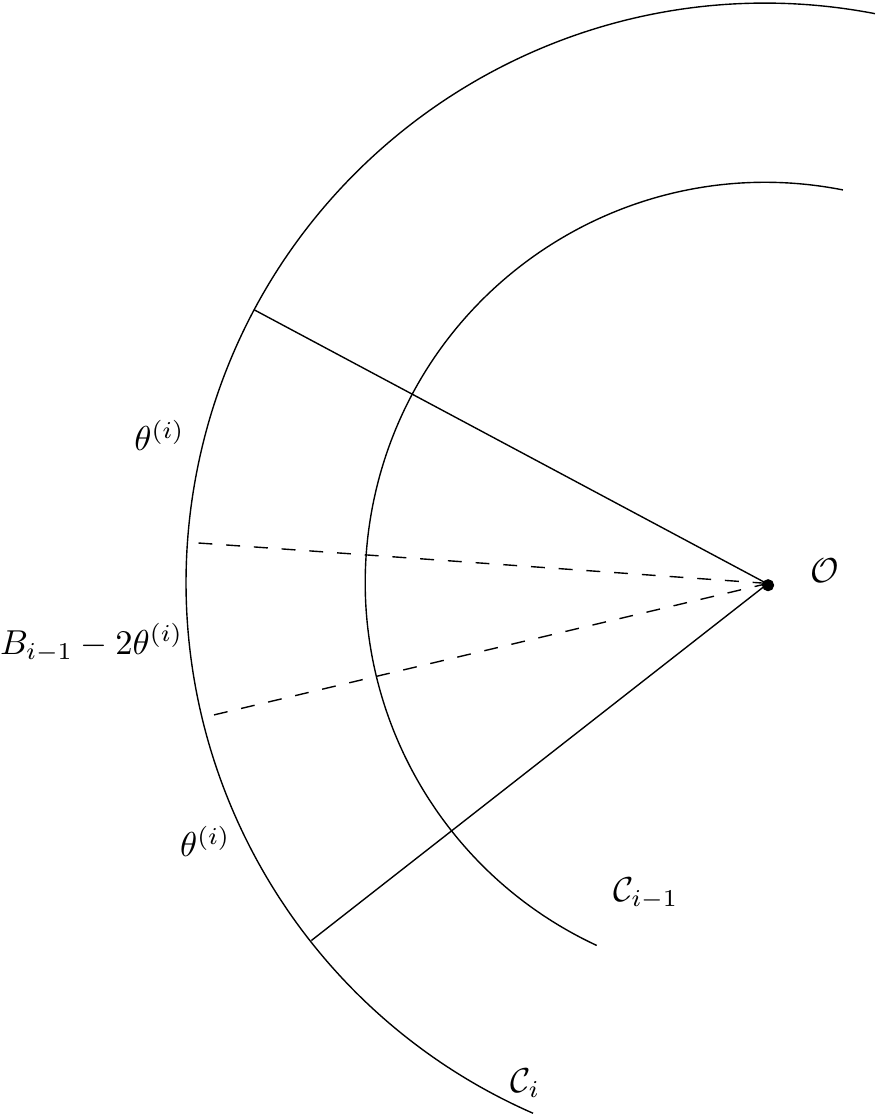}
\caption{Subdivisions of inherited blocks.}
\end{center}
\end{figure}

Now, we keep only the \emph{central} part of each block, discarding all the rest which as we will see later is negligible. 
We denote the collection of all remaining blocks of angle $B_{i-1}-2\theta^{(i)}$ by $ \mathcal{S}_{i-1}'$. 
Note that $|\mathcal{S}_{i-1}'| = |\mathcal{S}_{i-1}| = S_{i-1}$. 
We split each such part into
smaller \emph{uncolored} blocks of angle $B_i$. Finally, discard the blocks that are leftover (if any).
It is now clear that the number of (uncolored) blocks of angle $B_i$ is bounded from below by
\begin{equation}\label{eq:small_blocks}
|\mathcal{S}_{i-1}'|\left \lfloor \frac{B_{i-1}-2\theta^{(i)}}{B_i} \right  \rfloor = S_{i-1}\left \lfloor \frac{B_{i-1}}{B_i} -2t_i\right  \rfloor.
\end{equation}
Recall that $\Theta_{i-1}$ is the total angle covered by black blocks on the circle $\mathcal{C}_{i-1}$, that is, the blocks in
$\mathcal{S}_{i-1}$. We denote the blocks in $\mathcal{S}_{i-1}'$ by $I_j^{(i)}$, for $j=1,\ldots, K_i$, with $K_i := \lfloor  \Theta_{i-1}/B_i \rfloor$.

Analogously to (\ref{eq:def_active_area_1}), we define the active area below the circle $\mathcal{C}_i$ as
\[
\mathcal{A}^{(i)}:=\left \{ x=(r,\theta) \ : \ r_{i-1}<r<r_i, \ \theta \in I_j^{(i)} \textnormal{ for some }j\in \{1, \ldots , K_i\} \right \}.
\]
Given a block $I_j^{(i)}$ we define the following event:
\[
\begin{split}
\mathbf{B}_j^{(i)} & := \bigl \{ \textnormal{for some $k\geq \r$ there are vertices }x_1^{(i-1)}, x_2^{(i-1)}, \ldots ,x_k^{(i-1)}\in\mathcal{S}_{i-1},\\
& \textnormal{such that\ }I_j^{(i)} \textnormal{ is completely contained in the }k\textnormal{ disks of radius }R,\textnormal{ centered}\\
& \textnormal{at }x_1^{(i-1)}, x_2^{(i-1)}, \ldots ,x_k^{(i-1)},\textnormal{ AND at least $\r$  edges connecting each}\\
& x_j^{(i-1)}\textnormal{ with any vertex in }\mathcal{S}_{i-2}\textnormal{ are retained during the edge percolation process}\bigr \} . 
\end{split}
\]
Now we color each block $I_j^{(i)}$ \emph{black} if and only if the event $\mathbf{B}_j^{(i)}$ is realized. 
At this point we show a generalization of Lemma \ref{lemma:P_black_block_1}.
\begin{Lemma}\label{lemma:P_black_block_induction}
The expected value of black blocks on $\mathcal{C}_i$ is bounded from below by
\[
S_{i-1}\left (\frac{B_{i-1}}{B_i}-2t_i\right )(1-e^{-c\rho^\r t_i}),
\]
for some $c>0$ that does not depend on $i$. 
\end{Lemma}
\begin{proof}
By  the definition of the event $\mathbf{B}_j^{(i)}$, we first need to ensure that the interval $I_j^{(i)}$ is contained in the disk of
radius $R$ of $k\geq \r$ vertices, belonging to $\mathcal{S}_{i-1}$. We claim that for each vertex in $v \in \mathcal{B}_{i-1}$ this
occurs with probability at least $(\theta^{(i)} - B_i)/(2\pi)=\theta^{(i)}(1-t_i^{-1}) /(2\pi)$. 
%
%
%
To see this, note that by Lemma~\ref{lem:relAngle}, the intersection of the disk of radius $R$ around $v$ with $\mathcal{C}_i$ is an 
arc $I(v)$ of angle at least $2\theta^{(i)}$. 
Therefore, $I(v)$ covers a block $I_j^{(i)}$  for a range of the angle of $v$ which is at least $2\theta^{(i)} - B_i$ 
(as the length of $I_j^{(i)}$ is equal to $B_i$). 


Moreover, at least $\r$ edges must be
connecting each such vertex with a vertex in $\mathcal{S}_{i-2}$ and are retained after the edge percolation
process. Since this last event occurs with probability at least $\rho^{\r}$, we have that
\[
\begin{split}
\P(\mathbf{B}_j^{(i)})
& \geq \P\left ( \bin\left ( N_{i-1}, \frac{\theta^{(i)}(1-t_i^{-1})\rho^{\r}}{2\pi}\right )\geq \r \mid N_{i-1} \right )\\
& \geq 1-e^{-c t_i \rho^{\r}},
\end{split}
\]
which again follows from (\ref{eq:chernoff}), for some $c = c(\eps ) >0$, uniformly for all $i >0$. 

Denote by $\mathcal{S}_i$ the collection of black blocks that we end up with, and set $ S_i:=|\mathcal{S}_i|$. Thus, 
\begin{equation}\label{eq:def_L_i}
\E(S_i\mid S_{i-1},N_{i-1}) \geq  S_{i-1}\left ( \frac{B_{i-1}}{B_i}-2t_i\right )(1-e^{-c\rho^{\r} t_i})=: L_i,
\end{equation}
which concludes the proof.
\end{proof}
At this point we proceed as in Section \ref{sect:base_step} and show an analogue of Lemma \ref{lemma:concentration_S_1}.
\begin{Lemma}\label{lemma:concentration_S_i}
For every $i>0$ let $\ve_i:= (\theta^{(i)})^{1/6}$. For any $0< i < T$, conditional on $S_{i-1}$ and $N_{i-1}$, 
with probability at least $1-\exp{(-\Theta \left( 1/( t_i(\theta^{(i)})^{2/3}) \right)}$ we have
\[
S_i\geq L_i(1-\ve_i),
\]
where $L_i$ is defined in (\ref{eq:def_L_i}).
\end{Lemma}
\begin{proof}
Analogously to Lemma \ref{lemma:concentration_S_1}, we see that by changing the position of one vertex in $\mathcal{B}_{i-1}$, one can change the number of blocks in $\mathcal{S}_i$  by at most 
\[
2 \frac{\theta^{(i)}}{B_i} = 2\frac{\theta^{(i)}}{\theta^{(i)}/t_i}=2t_i.
\]
Hence, by (\ref{eq:bdd_ineq}) we get
\[
\begin{split}
& \P \bigl [S_i<(1-\varepsilon_i)\E(S_i\mid S_{i-1},N_{i-1})\mid S_{i-1},N_{i-1}\bigr ] 
\leq \exp{\left \{ -\frac{\ve_i^2 (\E(S_i\mid S_{i-1},N_{i-1}))^2}{N_{i-1}(2t_i)^2}\right \}}.
\end{split}
\]
We will specify $\ve_i$ so that the absolute value of the above exponent tends to infinity. 
Recall that 
\begin{equation}\label{eq:theta_S_B}
\Theta_{i-1} = S_{i-1} B_{i-1}.
\end{equation}
Furthermore, if we demand that  for $i\leq T$ we have $\Theta_{i-1} > \pi/2 $ (see the definition of $T$ in the next section), then 
the absolute value of the above exponent is bounded as follows: 
\[
\begin{split}
\frac{\ve_i^2 \E^2(S_i\mid S_{i-1},N_{i-1})}{N_{i-1}(2t_i)^2}
& \gtrsim  \frac{\ve_i^2 (S_{i-1})^2(B_{i-1}/B_i)^2}{N_{i-1}(2t_i)^2} \stackrel{(\ref{eq:theta_S_B})}{\asymp} 
\frac{\ve_i^2 \Theta_{i-1}^2}{N_{i-1}t_i^2 B_i^2}\\
& \asymp \frac{\ve_i^2 }{N_{i-1}(t_i)^2B_i^2}
=\frac{\ve_i^2}{N_{i-1}(\theta^{(i)})^2}.
\end{split}
\]
By (\ref{eq:theta^ixN_i-1}) we get 
\[
\frac{\ve_i^2}{N_{i-1}\theta^{(i)}\theta^{(i)}}\asymp \frac{\ve_i^2}{t_i\theta^{(i)}}.
\]
Hence, analogously to the base case, we get that $\ve_i$ should satisfy
\begin{equation*}
\ve_i\gg \sqrt{t_i \theta^{(i)}}.
\end{equation*}
We choose
\begin{equation*}
\ve_i:=\left (\theta^{(i)}\right )^{1/6},
\end{equation*}
and set
\begin{equation}\label{eq:def_M2}
M_i^{(2)}(N):=\frac{\ve_i^2}{t_i \theta^{(i)}}=\frac{1}{t_i (\theta^{(i)})^{2/3}}\,, \quad \forall i\geq 1.
\end{equation}
In other words, we apply Lemma \ref{lemma:P_black_block_induction} and get that with probability at least 
$1-e^{-\Theta ( M_i^{(2)}(N) )}$ we have
\[
S_i \geq L_i(1-\ve_i)=: L_i', 
\]
where $L_i$ has been defined in (\ref{eq:def_L_i}).
The proof of the fact that $M_i^{(2)}(N)\to \infty $ is deferred to Appendix \ref{sect:linearly_many_vertices}, hence the proof of the Lemma is concluded.
\end{proof}
This in turn implies that conditional on $S_{i-1}$, with probability at least $1-e^{-\Theta (M_i^{(2)}(N))}$, the total length $\Theta_i $ of the set 
of blocks in $\mathcal{S}_i$ is bounded from below by
\begin{equation*}
L_i'B_i = S_{i-1}(B_{i-1}-2B_i t_i)(1-e^{-c \rho^{\r} t_i})\left (1-\left (\theta^{(i)}\right )^{1/6}\right ).
\end{equation*}
By putting together these facts, we deduce the following result.
\begin{Proposition}\label{prop:Theta_i>pi/2}
For any $1< i < T$,  conditional on $\Theta_{i-1}$, we have
$$\Theta_i \geq  \Theta_{i-1}\left ( 1-\left ( 2\frac{\theta^{(i)}}{\theta^{(i-1)}}t_{i-1}+e^{-c \rho^{\r} t_i}
+\left ( \theta^{(i)}\right )^{1/6} \right) \right),$$
with probability at least $1- \exp \left( - \Theta ( M_i^{(2)} (N) )\right)$. 
\end{Proposition}
\begin{proof}
By Lemma \ref{lemma:concentration_S_i}, with high probability we have
\begin{equation}\label{eq:first_lower_Theta}
\begin{split}
\Theta_i & = S_iB_i
\geq L_i'B_i=S_{i-1}B_{i-1}\left (1-2\frac{B_i}{B_{i-1}} t_i\right )(1-e^{-c\rho^{\r} t_i})\left (1-\left (\theta^{(i)}\right )^{1/6}\right )\\
& =S_{i-1}B_{i-1}\left ( 1-2\frac{\theta^{(i)}}{\theta^{(i-1)}}t_{i-1}\right )(1-e^{-c \rho^{\r} t_i})\left (1-\left (\theta^{(i)}\right )^{1/6}\right )\\
& = \Theta_{i-1}\left ( 1-2\frac{\theta^{(i)}}{\theta^{(i-1)}}t_{i-1}\right )(1-e^{-c\rho^{\r} t_i})\left (1-\left (\theta^{(i)}\right )^{1/6}\right )\\
& \geq \Theta_{i-1}\left ( 1-\left ( 2\frac{\theta^{(i)}}{\theta^{(i-1)}}t_{i-1}+e^{-c \rho^{\r} t_i}+\left (\theta^{(i)}\right )^{1/6}\right )
\right ),
\end{split}
\end{equation}
which concludes the proof.
\end{proof}

\subsection{Proof of Theorem \ref{thm:main}}\label{sect:proof_thm_main}
\subsubsection{Parts $(i)$ and $(ii)$}
Given a small value $0<\ve<1$, we choose a suitable large constant $C = C(\alpha, \nu, \ve) $ (which will be defined explicitly in Appendix \ref{sect:step1}) and we set 
\begin{equation}\label{eq:definition_T1_T2}
T_1 := \min \{ i \ : \  t_i < C\},\quad  T_2:=\min \{ i \ : \ \Theta_i<\pi/2\}.
\end{equation}
We take
\begin{equation}\label{eq:min_T1_T2}
T:=\min \{ T_1,T_2\}.
\end{equation}
Let us denote by $N_i'$ the number of vertices in $\mathcal{B}_i$ which belong to $\K$.
The random variable $N_i'$ \emph{stochastically dominates} a
binomial random variable. More specifically, with $\succcurlyeq$ denoting stochastic domination,  conditional on $N_i$ and $\Theta_i$, we
have
\begin{equation}\label{eq:preview_N_i'}
N_i'\succcurlyeq \bin \left ( N_i, \frac{\Theta_i}{2\pi}\rho^{\r}\right ).
\end{equation}
This is the case because each vertex in $\mathcal{B}_i$ falls inside $\mathcal{A}^{(i)}$ with probability at least $\Theta_i / (2\pi)$ and, given 
this, it is connected to at least $\r$ of the vertices in $\mathcal{A}^{(i-1)}$ with probability at least $\rho^{\r}$. Furthermore, these events are  
independent for the set of vertices in $\mathcal{B}_i$, whereby (\ref{eq:preview_N_i'}) follows. 

For $i<T\leq T_2$ (cf. Definitions (\ref{eq:definition_T1_T2}), (\ref{eq:min_T1_T2})), the stochastic inequality (\ref{eq:preview_N_i'})
implies that 
\[
\E \left (N_i' \mid N_i,\Theta_i\right )\geq N_i \frac{\Theta_i}{2\pi}\rho^{\r}\geq N_i\frac{\rho^{\r}}{4}.
\]
Now, for any value $0<\delta<1 $, we can apply a standard Chernoff bound (\ref{eq:chernoff}), which leads to
\begin{equation}\label{eq:P_azuma}
\P \left (N_i'<(1-\delta)N_i\frac{\rho^{\r}}{4} \mid N_i, \Theta_i \right )\leq e^{-\delta^2 N_i\rho^{\r}/8}.
\end{equation}
Hence conditional on $\mathcal{E}_i$ (defined in (\ref{eq:N})), for any $\delta \in (0,1)$, we have that a.a.s.
\begin{equation}\label{eq:N_i'}
N_i'\geq (1-\delta)N_i\frac{\rho^{\r}}{4}.
\end{equation}
Each of these vertices will be connected to the previous band by at least $\r$ edges, hence we conclude showing that 
there is a positive constant $\kappa=\kappa(\alpha,C,\ve,\delta,\rho,\r)$ for which
\begin{equation}\label{eq:sum_N_i'}
|\K| \geq \sum_{i=0}^{T-1} N_i'\geq \kappa N.
\end{equation}
Let us set
\[
M_i^{(1)}(N):=\delta^2 N_{i}\rho^{\r}/8.
\]
Hence, by choosing $\delta$ to be an arbitrarily small constant, with probability at least
\begin{equation}\label{eq:lower_bd_probability}
1-\sum_{i=1}^{T-1}\left ( e^{-M_i^{(1)}(N)}+e^{-M_i^{(2)}(N)}\right ),
\end{equation}
(where $M_i^{(2)}(N)$ was defined in (\ref{eq:def_M2})) we have for $0<i<T$
\begin{equation*}
N_{i}'\geq (1-\delta)N_{i}\frac{\rho^{\r}}{4}.
\end{equation*}
Thus, the above implies that
\[
\begin{split}
\sum_{i=1}^{T-1} N_{i}'
& \geq \sum_{i=1}^T N_{i}(1-\delta)\frac{\rho^{\r}}{4}\geq \sum_{i=1}^{T-1} \E N_i (1-\ve)(1-\delta)\frac{\rho^{\r}}{4}\\
& \stackrel{Claim~\ref{claim:Ni}}{\geq} (1-\ve)^3(1-\delta)\sum_{i=1}^{T-1} N(e^{-\alpha t_i}-e^{-\alpha t_{i-1}})\frac{\rho^{\r}}{4}\\
& \geq (1-\ve)^3 (1-\delta)\frac{\rho^{\r}}{4}\sum_{i=1}^{T-1} Ne^{-\alpha t_i}\left (1-e^{-\alpha (t_{i-1}-t_i)}\right )\\
&  \stackrel{(\ref{eq:T_C_1})}{\geq} (1-\ve)^4 (1-\delta)\frac{\rho^{\r}}{4} N\sum_{i=1}^{T-1} e^{-\alpha t_i}.\\
\end{split}
\]
Recall from Section \ref{sect:C2_from_other_paper} that $\lambda=2\alpha-1$. 
In Appendix~\ref{sect:step1}, we show that 
\begin{equation} \label{eq:sum_exps}
\sum_{i=1}^{T-1} e^{-\alpha t_i} \geq e^{\alpha C/\lambda}.
\end{equation}
Setting, for example,
\[
\kappa=\kappa(\alpha,C,\ve,\delta,\rho,\r):=\frac{(1-\ve)^4 (1-\delta) \rho^{\r}}{4} e^{-\alpha C/ \lambda},
\]
we deduce (\ref{eq:sum_N_i'}), concluding the proof of parts \emph{(i)} (a.a.s.\ case) and \emph{(ii)} (with positive probability case).

\subsubsection{Part $(iii)$}
In this case it suffices to show that if $\varphi(N)=p(N)N^{1/2\alpha}=o(1)$ and $\r \geq 2$, then
\begin{equation}\label{eq:prob_0}
\P\bigl (\textnormal{there exists }v\in \V \textnormal{ has at least }\r\textnormal{ initially infected neighbors}\bigr )=o(1).
\end{equation}
The proof is based on the fact that for any vertex the probability of having at least $\r\geq 2$ infected neighbors after the first round is $o(N^{-1})$.

More precisely, we resort to Lemma \ref{lemma:deg_distrib}, which states that the degree $\mathbf{d} (v)$ of vertex $v\in \V$ (conditional on the type) is such that 
\begin{equation}\label{eq:dv_sum_ber}
\mathbf{d} (v) \preccurlyeq \sum_{\ell =1}^N \operatorname{Ber}\left (\frac{H_2 e^{t_v/2}}{N} \right ),
\end{equation}
where $H_2>0$ is as in Lemma \ref{lemma:deg_distrib}, and $\operatorname{Ber} (p)$ denotes a Bernoulli random variable with parameter $p$.
At this point we can distinguish between two cases: either $ t_v\geq R/10$, or $ t_v< R/10$, hence divide the set $\V$ into two disjoint sub-sets:
\[
\mathfrak{D}_3:=\left \{v\ : \ R/10\leq t_v \leq R/2\alpha+\omega(N)\right \},
\]
and
\[
\mathfrak{D}_4:=\left \{v\ : \ 0<t_v< R/10\right \}.
\]
In the first case, by Lemma \ref{lemma:deg_distrib} we have 
\begin{equation}\label{eq:lower_bd_dv}
\begin{split}
\P (\mathbf{d} (v) \geq 2H_2 e^{t_v/2} \ | \ t_v)
& \leq \P\left (\bin \left (N, \frac{H_2e^{t_v/2}}{N}\right )\geq 2H_2 e^{t_v/2}\right )\\
& \stackrel{(\ref{eq:chernoff_up})}{\leq } \exp \left (-\frac{H_2 e^{t_v/2}}{8}\right )=o(N^{-1}).
\end{split}
\end{equation}
Now we proceed as follows: to simplify the notation in the next calculation, denote by $ \mathfrak{in}(v)$ the number of initially
infected neighbors of $v$. Hence
\[
\begin{split}
\P & (v\in \mathfrak{D}_3 \ \mbox{and}\  \mathfrak{in}(v) \geq \r) =\int_{R/10}^{R/2\alpha+\omega(N)}\P (\mathfrak{in}(v)\geq \r\mid t_v)\overline{\rho} (t_v)dt_v\\
& = \int_{R/10}^{R/2\alpha+\omega(N)} \P(\mathfrak{in}(v)\geq \r \mid t_v, \mathbf{d}(v)\geq 2H_2 e^{t_v/2})\P(\mathbf{d}(v)\geq 2H_2 e^{t_v/2} \mid t_v)\overline{\rho} (t_v)dt_v\\
& \quad + \int_{R/10}^{R/2\alpha+\omega(N)} \P(\mathfrak{in}(v)\geq \r \mid t_v, \mathbf{d}(v) < 2H_2 e^{t_v/2})\P(\mathbf{d} (v)< 2H_2 e^{t_v/2} \mid t_v)\overline{\rho} (t_v)dt_v\\
& \leq \int_{R/10}^{R/2\alpha+\omega(N)} \P(\mathbf{d}(v)\geq 2H_2 e^{t_v/2} \mid t_v)\overline{\rho} (t_v)dt_v \\
& \quad +\int_{R/10}^{R/2\alpha+\omega(N)} \P(\mathfrak{in}(v)\geq \r \mid t_v, \mathbf{d}(v)< 2H_2 e^{t_v/2})
\overline{\rho}(t_v)dt_v\\
& \stackrel{(\ref{eq:lower_bd_dv}) \textnormal{ and Lemma } \ref{lemma:deg_distrib}}{\leq } o(N^{-1})+ \int_{R/10}^{R/2\alpha+\omega(N)}\P\left (\bin \left (2H_2 e^{t_v/2},p(N)\right )\geq \r\mid t_v\right )\overline{\rho}(t_v)dt_v.
\end{split}
\]
Recall that $\r\geq 2$. 
Hence we obtain
\[
\begin{split}
& 
\int_{R/10}^{R/2\alpha+\omega(N)}\P\left (\bin \left (2H_2 e^{t_v/2},p(N)\right )\geq \r\mid t_v\right )\overline{\rho} (t_v)dt_v\\
& \lesssim \int_{R/10}^{R/2\alpha+\omega(N)}\left ( 2H_2e^{t/2}p(N)\right )^\r e^{-\alpha t}dt \\
& \asymp p(N)^\r \int_{R/10}^{R/2\alpha+\omega(N)}e^{(\r/2-\alpha)t}dt \stackrel{\r \geq 2}{\asymp} p(N)^\r e^{(\r/2-\alpha)(R/2\alpha+\omega(N))}\\
& \asymp \left ( p(N)N^{1/2\alpha}\right )^\r N^{-1}e^{(\r/2-\alpha)\omega(N)} 
\asymp \varphi(N)^\r N^{-1}e^{(\r/2-\alpha)\omega(N)}. 
\end{split}
\]
By choosing $\omega(N)$ such that
\[
\varphi(N)=o(e^{-(1/2- \alpha/ \r)\omega(N)}),
\]
then the above calculation leads to 
\[
\int_{R/10}^{R/2\alpha+\omega(N)}\left ( 2H_2e^{t/2}p(N)\right )^\r e^{-\alpha t}dt\lesssim \varphi(N)^\r N^{-1}e^{(\r/2-\alpha)\omega(N)} = o(N^{-1}).
\]
Therefore we have
\[
\P (v\in \mathfrak{D}_3\ \mbox{and} \ \mathfrak{in}(v) \geq \r )=o(N^{-1}).
\]
Now we take care of the vertices $v\in \mathfrak{D}_4$. 
For simplicity of notation, we set
\begin{equation}\label{eq:def_m}
\mathbf{m}:=H_2 e^{t_v/2}.
\end{equation}
By Remark \ref{remark:K_2>1} we have that for every $\r\geq 2$
\begin{equation}\label{eq:m>1}
\mathbf{m} \geq H_2 >1.
\end{equation}
At this point on (\ref{eq:dv_sum_ber}) we apply Le Cam's Theorem, and conditional on $t_v$ we have 
\[
d_{\operatorname{TV}}(\mathbf{d} (v),\hat{D}_v) \leq 2 N\left (\frac{H_2 e^{t_v/2}}{N} \right )^2,
\]
where $\hat{D}_v$ denotes a random variable following a Poisson distribution with parameter $\mathbf{m}$, and $d_{\operatorname{TV}}$ is the total variation distance.
%
Hence we obtain
\begin{equation}\label{eq:d_TV}
d_{\operatorname{TV}}(\mathbf{d} (v),\hat{D}_v)\leq 2\frac{H_2^2 e^{t_v}}{N}\leq 2 \frac{H_2^2 e^{R/10}}{N}=o(1).
\end{equation}
This immediately implies that $\mathbf{d}(v)$ is, with a very good approximation (as $N\to \infty$), distributed like $ \hat{D}_v$.
Furthermore, inequality (\ref{eq:chernoff_up}) implies that for every $v\in \mathfrak{D}_4$ we have
\begin{equation}\label{eq:upper_bd_degree}
\P(\mathbf{d}(v)\geq 2H_2e^{t_v/2}\ | \ t_v)\leq \exp(-H_2 e^{R/20}/4).
\end{equation}
From now on we shall be conditioning on the event that every $v\in \mathfrak{D}_4$ has degree at most 
$2H_2e^{t_v/2} \stackrel{t_v \leq R/10}{\leq} 2H_2e^{R/20}$, which by (\ref{eq:upper_bd_degree}) occurs with probability 
$1-o \left(\frac{1}{N} \right)$.

To simplify the next calculation, set
\[
\mathfrak{n}:=2H_2e^{R/20}.
\]
This leads to
\[
\begin{split}
\P & \left ( \bin \left ( \mathbf{d}_v, p(N)\right )\geq \r \mid t_v\right )  = \sum_{\ell =\r}^\mathfrak{n} \P\left ( \bin \left ( \ell , p(N)\right )\geq \r \mid \mathbf{d}_v=\ell ,t_v\right ) \P( \mathbf{d}_v=\ell \mid t_v)\\
& \leq \sum_{\ell =\r}^\mathfrak{n} \left ( \ell  p(N)\right )^\r \left ( \P( \hat{D}_v=\ell)+\P( \mathbf{d}_v=\ell \mid t_v)-\P( \hat{D}_v=\ell)\right )\\
& \stackrel{(\ref{eq:d_TV})}{\leq} \sum_{\ell =\r}^\mathfrak{n}  \left ( \ell  p(N)\right )^\r \P( \hat{D}_v=\ell )+\sum_{\ell =\r}^\mathfrak{n} \left ( \ell  p(N)\right )^\r \left (\frac{H_2^2 e^{R/10}}{N}\right )\\
& \lesssim \sum_{\ell =\r}^\mathfrak{n}  p(N)^\r \left (  \ell^\r e^\mathbf{-m}\frac{\mathbf{m}^\ell}{\ell!}\right )+ \mathfrak{n}^{\r+1} p(N)^\r H_2^2 e^{-2R/5}.
\end{split}
\]
Now it is easy to see that there is a constant $K:=K(\r)>0$ such that for all $\ell\in \N$ we have
\[
\ell^\r \frac{\mathbf{m}^\ell}{\ell!}\leq K \frac{d^\r}{d\mathbf{m}^\r}\frac{\mathbf{m}^\ell}{\ell !}.
\]
Therefore we obtain
\[
\begin{split}
\P\left ( \bin \left ( \mathbf{d}_v, p(N)\right ) \geq \r \mid t_v\right )
& \lesssim K p(N)^\r e^\mathbf{-m}\sum_{\ell =\r}^\infty \frac{d^\r}{d\mathbf{m}^\r}\frac{\mathbf{m}^\ell}{\ell !}+\mathfrak{n}^{\r+1} p(N)^\r H_2^2 e^{-2R/5}\\
& \lesssim K p(N)^\r e^\mathbf{-m} \frac{d^\r}{d\mathbf{m}^\r}\sum_{\ell =0}^\infty \frac{\mathbf{m}^\ell}{\ell !}+(e^{R/20})^{\r+1} p(N)^\r e^{-2R/5}\\
& = K p(N)^\r e^\mathbf{-m}e^\mathbf{m}+e^{(\r+1)R/20} p(N)^\r e^{-2R/5}. 
\end{split}
\]
Now we show that $e^{(\r+1)R/20} p(N)^\r e^{-2R/5}=o(N^{-1})$.
This is easy to see, since $\r\geq 2$. In fact we have:
\[
\begin{split}
e^{(\r+1)R/20} p(N)^\r e^{-2R/5}
& \asymp N^{(\r+1)/10}N^{\r/2\alpha-\r/2\alpha} p(N)^\r N^{-4/5}\\
& =\varphi(N)^\r N^{\r(1/10-1/2\alpha)+1/10-4/5} \\
& \stackrel{\alpha<1}{<}\varphi(N)^\r N^{-2\r/5-7/10} ~ \stackrel{\r\geq 2}{<}~\varphi(N)^\r N^{-15/10}=o(N^{-1}).
\end{split}
\]
Hence, we can write
\[
\P (v\in \mathfrak{D}_4\ \mbox{and} \ \mathfrak{in}(v) \geq \r \mid t_v) \leq K p(N)^\r +o(N^{-1}).
\]
By integrating over the types and reasoning similarly as in  the case of $v\in \mathfrak{D}_3$ we have
\[
\begin{split}
& \P (v\in \mathfrak{D}_4\ \mbox{and} \ \mathfrak{in}(v) \geq \r )  \asymp \int_0^{R/10} p(N)^\r e^{-\alpha t}dt +o(N^{-1})\\
& \asymp p(N)^\r +o(N^{-1}) .
\end{split}
\]
But $p(N) \ll N^{-1/(2\alpha)}$. 
Moreover, since  $1/2 < \alpha <1$ we have $2\alpha < 2$. So as $\r \geq 2$ it follows that 
$\frac{\r}{2\alpha} >1$, whereby $p(N)^{\r} \ll N^{-1}$.
Hence, 
\[
 \P (v\in \mathfrak{D}_4\ \mbox{and} \ \mathfrak{in}(v) \geq \r ) = o(N^{-1}).
\]
The fact that this probability decreases asymptotically faster than $N^{-1}$ implies that no vertex becomes infected during the first round.
In other words, $|\mathcal{A}_f |=|\mathcal{A}_0|$ a.a.s.\ which concludes the proof of Theorem \ref{thm:main}.

\section{Proof of Corollary~\ref{cor:giant} and Theorem~\ref{cor:core}}\label{sect:proof_corollaries}
In this section we show how to obtain Corollary \ref{cor:giant} and Theorem \ref{cor:core} from Theorem \ref{thm:main}.
\begin{proof}[Proof of Corollary \ref{cor:giant}]
To show that there is a giant component, it suffices to set $\r=1$, and apply Theorem \ref{thm:main} with a high infection rate.
In fact, if there is at least one infected vertex in a connected component, then the bootstrap percolation process will eventually 
infect the whole component.
In this case, we can assume that $p=p(N)$ is a positive constant. 
This implies that a.a.s.\ at least 1 vertex in $\mathcal{B}_0$ is initially infected and as the graph induced by these vertices is
complete, it follows that all these vertices become infected during the second round. 
But by (\ref{eq:sum_N_i'}), it follows that a.a.s. 
$$ |\mathcal{K}_1 (\mathcal{B}_0)| \geq \kappa N, $$
for some $\kappa > 0$. Hence, the largest component of $\mathcal{G}(N;\alpha, \nu, \rho)$ has at least $\kappa N$ vertices. 
\end{proof}
\begin{proof}[Proof of Theorem~\ref{cor:core}]
The proof of this theorem is a byproduct of our proof. In fact, $\mathcal{K}_{\r} (\mathcal{B}_0)$ by its construction is a subgraph 
of minimum degree at least $\r$.
%
\end{proof}

\appendix

\section{The definition of $T$ and the proof of Lemma \ref{lemma:T=O(lnR)}} \label{sect:step1}
The first step to show Lemma \ref{lemma:T=O(lnR)} is given by the following result.
\begin{Claim}\label{claim:theta_i-t_i}
If $i\geq 1$ is such that
\begin{equation}\label{eq:assumption_i}
1- \alpha >\frac{2}{t_i}\ln \left ( \frac{4\pi}{(1-\ve)^4}t_i\right ),
\end{equation}
then
\[
t_i<\alpha t_{i-1}.
\]
Therefore
\[
t_i< \alpha^i\frac{R}{2}.
\]
\end{Claim}
\begin{proof}
%
Recall the definition of $t_i$ from (\ref{eq:def_ti}) and that $ \lambda=2\alpha-1$. 
Notice that Condition (\ref{eq:assumption_i}) can be rewritten as
\begin{equation*} 
2 \ln \left( \frac{4\pi}{\nu(1-\ve)^4}  t_i \right) < \left( 1 - \alpha \right) t_i. 
\end{equation*}
This condition implies that 
\begin{equation*} 
\begin{split}
t_i &= \lambda t_{i-1} + 2 \ln \left( \frac{4\pi}{\nu(1-\ve)^4} t_i \right) \stackrel{t_i \leq t_{i-1}}{<} 
\lambda t_{i-1} + 2 \ln \left( \frac{4\pi}{\nu(1-\ve)^4}  t_{i-1} \right) \\ 
&< \left( \lambda + 1 - \alpha \right) t_{i-1} = 
\left( 2\alpha - 1 + 1 - \alpha\right) t_{i-1} = \alpha t_{i-1}.
\end{split}
\end{equation*}
\end{proof}
Now we make the definition of $T$ more precise.
In particular, recalling Equations (\ref{eq:min_T1_T2}) and (\ref{eq:definition_T1_T2}), here we specify the constant $C=C(\alpha, \nu, \ve)$,
which must large enough so that all the following relations are satisfied:
\begin{align}
\label{eq:T_C_1}&e^{-C \alpha(1- \alpha)} < \ve, & \\ 
\label{eq:T_C_2}&\mbox{if } x - 2 \ln \left({4 \pi \over \nu (1-\ve)^4} x \right) \geq \lambda C, \mbox{ then } \frac{\nu(1-\ve)^4}{4\pi} < x,
&\\
\label{eq:T_C_3}&{2\over \lambda C}  \ln \left( {4 \pi \over \nu (1-\ve)^4} {\lambda C} \right)< {1 - \lambda \over 2},&\\
\label{eq:T_C_4}&{e^{-c \rho^{\r} \lambda C} \over 1- e^{-c \rho^{\r} \left(1- \alpha \right)C}} < {1\over 8},&\\
\label{eq:T_C_5}& \int_{\frac{(1-\eta^2)}{2}C}^\infty xe^{-x}dx< \frac{1}{16}\frac{(1-\eta^2)}{2\eta},& \\
\label{eq:T_C_6}& C > \max\left\{{4 \over \lambda}~{\alpha \over 1-\alpha^2}, {2\alpha \over (1-\alpha)^2 (1+\alpha)} \right\}.
\end{align}
Let us see now what some of these conditions imply (the rest will become clear in the next few pages).
Claim \ref{claim:theta_i-t_i} implies that 
\begin{equation*} 
e^{- \alpha (t_{i-1} - t_i)} < e^{-\alpha t_{i-1} \left(1 - \alpha \right)}. 
\end{equation*}
Thus, (\ref{eq:T_C_1}) implies that for any $1\leq i < T$ we have 
\begin{equation}  \label{eq:T_1} 
e^{- \alpha (t_{i-1} - t_i)} \stackrel{t_{i-1} \geq C}{<} \ve. 
\end{equation}
Condition (\ref{eq:T_C_2}) is used in order to ensure that for all $1\leq i \leq T$ we have $t_i>\lambda^i t_0$.
Indeed, using (\ref{eq:def_ti}) we have that for any $1\leq i\leq T$ 
\[
\lambda C \leq \lambda t_{T-1}\leq \lambda t_{i-1} = t_i - 2 \ln \left( \frac{4\pi}{\nu(1-\ve)^4}  t_i \right).
\]
But then Condition (\ref{eq:T_C_2}) implies that $t_i > \frac{\nu(1-\ve)^4}{4\pi}$, whereby 
$\ln \left( \frac{4\pi}{\nu(1-\ve)^4}  t_i \right) > 0$. In turn, 
\begin{equation} \label{eq:rec_low}
t_i > \lambda t_{i-1} \geq \lambda C, \  \mbox{for any}\ i\leq T. 
\end{equation}
Condition (\ref{eq:T_C_3}) together with the previous observation ($t_i > \lambda C$, for any $i\leq T$) imply that the hypothesis of 
Claim~\ref{claim:theta_i-t_i} holds for all $i\leq T$. 
Hence, $T = O (\log R)$ as in Lemma~\ref{lemma:T=O(lnR)}

\medskip

\subsubsection*{Proof of (\ref{eq:sum_exps})}
As shown in Proposition~\ref{prop:Theta_i>pi/2}, for all $0< i <T $ with probability bounded from below by 
$1-\exp{(-\Theta \left( M_i^{(2)} (N) \right)}$ we have
\begin{equation*}
\Theta_i \geq \Theta_{i-1}\left ( 1-\left ( 2\frac{\theta^{(i)}}{\theta^{(i-1)}}t_{i-1}+e^{-c \rho^{\r} t_i}+\bigl (\theta^{(i)}\bigr )^{1/6}\right )
\right ).
\end{equation*}
Now note that if this event is realized for all $j \leq i < T$, we deduce that 
\[ 
\begin{split}
\Theta_i & \geq \Theta_1 \prod_{j=2}^i \left ( 1-\left ( 2\frac{\theta^{(j)}}{\theta^{(j-1)}}t_{j-1}+e^{-c \rho^{\r} t_j}+\bigl ( \theta^{(i)}\bigr )^{1/6}\right )\right )\\
& \geq \Theta_1 \left ( 1- \sum_{j=2}^{T-1} \left ( 2\frac{\theta^{(j)}}{\theta^{(j-1)}}t_{j-1}+e^{-c \rho^{\r} t_j}+\bigl ( \theta^{(j)}
\bigr )^{1/6}\right )\right ).
\end{split}
\] 
In the next section we use Conditions (\ref{eq:T_C_4})--(\ref{eq:T_C_6}) to show that for $N$ is sufficiently large we have
\begin{equation}\label{eq:fudamental_relations}
\sum_{i=2}^{T-1}\frac{\theta^{(i)}}{\theta^{(i-1)}}t_{i-1}<\frac{1}{16}; \quad \textnormal{and} \quad \sum_{i=2}^{T-1}e^{-c \rho^{\r} t_i}<\frac{1}{8}; \quad\textnormal{and} \quad \sum_{i=2}^{T-1}\left (\theta^{(i)}\right )^{1/6}<\frac{1}{8}.
\end{equation}
Hence, for all $0\leq i < T_1$ (with $T_1$ defined in (\ref{eq:definition_T1_T2}))
\begin{equation}\label{eq:lower_bound_Theta_i}
\Theta_i \geq \Theta_1 \left (1-\frac{3}{8}\right )> \frac{\pi}{2},
\end{equation}
which in turn implies that in relation (\ref{eq:min_T1_T2}) we have $T=\min\{T_1,T_2\}=T_1$.
Moreover, this occurs with probability at least $1 - \sum_{i=2}^{T-1} \exp \left(- \Theta (M_i^{(2)}(N)) \right)$. 
If $T=T_1$, then (\ref{eq:rec_low}) implies that 
\[
\sum_{i=1}^{T-1} e^{-\alpha t_i} \geq \sum_{i=1}^{T-1} e^{-\alpha {1 \over \lambda^i} t_T} \geq  
\sum_{i=1}^{T-1} e^{-\alpha {1 \over \lambda^i} C} \geq 
e^{-\alpha C/\lambda}.
\]

\section{Bounds on the error terms in (\ref{eq:fudamental_relations})}\label{sect:steps}
\paragraph{Bound on first error term}
Here we show that $\sum_{i=2}^T \frac{\theta^{(i)}}{\theta^{(i-1)}}t_{i-1}<1/16$, when $N$ is large enough.
First note that 
$$\sum_{i=2}^T \frac{\theta^{(i)}}{\theta^{(i-1)}}t_{i-1} = \sum_{i=2}^{T} t_{i-1}e^{(t_i-t_{i-2})/2} 
= \sum_{i=1}^{T-1} t_{i}e^{(t_{i+1}-t_{i-1})/2}.$$
By Claim~\ref{claim:theta_i-t_i}, we have $t_{i+1} < \alpha t_i$ and $t_{i-1} > t_i/\alpha$. Thereby, 
$t_{i+1} - t_{i-1} < t_i (\alpha - 1/\alpha)$. So we have 
\[
\sum_{i=2}^T \frac{\theta^{(i)}}{\theta^{(i-1)}}t_{i-1} \leq \sum_{i=1}^{T-1} t_{i} e^{-{t_i \over 2} 
\left( {1\over \alpha} - \alpha \right)} = {2\alpha \over 1 -\alpha^2 } 
\sum_{i=1}^{T-1} {1\over 2}\left( {1\over \alpha} - \alpha \right) t_{i} e^{-{t_i \over 2} 
\left( {1\over \alpha} - \alpha \right)}.
\]
Now, the last sum can be bounded as follows: 
\begin{equation*} 
\sum_{i=1}^{T-1} {1\over 2}\left( {1\over \alpha} - \alpha \right) t_{i} e^{-{t_i \over 2} 
\left( {1\over \alpha} - \alpha \right)} 
\leq \int_{{1\over 2}\left( {1\over \alpha} - \alpha \right) t_T - 1}^\infty x e^{-x}dx 
\leq \int_{{1\over 2}\left( {1\over \alpha} - \alpha \right) \lambda C - 1}^\infty x e^{-x}dx. 
\end{equation*}
Let us set $\eta : = {1\over 2}\left( {1\over \alpha} - \alpha \right)$.
Indeed, this is bound holds provided that $C$ is large enough so that for any 
$x > \eta \lambda C - 1$, the function 
$xe^{-x}$ is decreasing, namely for $x>1$, and, moreover, $\eta(t_{i-1} - t_i) \geq 1$. 

In particular, the former holds if 
$C > {4 \over \lambda}~{\alpha \over 1-\alpha^2}$, which is
implied by Condition (\ref{eq:T_C_6}). 

Moreover, we have 
\[
t_{i-1} - t_i > (1-\alpha) t_{i-1} > (1-\alpha) C
\]
So if $\eta (1-\alpha ) C >1$, then $\eta(t_{i-1} - t_i) \geq 1$ and the approximation of the sum by an integral is valid. 
This condition is $C > {1 \over \eta (1-\alpha)} = {2\alpha \over (1-\alpha)^2 (1+\alpha)}$, which is again implied by (\ref{eq:T_C_6})
Therefore, bounding the sum by the integral is valid.

\paragraph{Bound on the second error term}
Now, we verify that
\[
\sum_{j=2}^{T} e^{-c \rho^{\r} t_i} < {1\over 8}.
\]
As $t_i - t_{i-1} < (\alpha - 1)C$, we can write
\[
\sum_{j=2}^{T} e^{-c \rho^{\r} t_j} < e^{-c \rho^{\r} t_T} \sum_{j=0}^\infty e^{-j c \rho^{\r}\left( 1- \alpha \right) C}.
\]
Also, as we have shown above $ t_T > \lambda C$.
Therefore, 
\[
e^{-c\rho^{\r} t_T} \sum_{j=0}^\infty e^{-j c \rho^{\r}\left( 1- \alpha \right) C} < {e^{-c \rho^{\r} \lambda C} \over 1- e^{-c \rho^{\r} 
\left(1- \alpha \right)C}} \stackrel{(\ref{eq:T_C_4})}{<} {1\over 8}.
\]

\paragraph{Bound on the third error term}
Now we need to check that, when $N$ is large enough, we get
\[
\sum_{i=2}^T \left ( \theta^{(i)}\right )^{1/6} < 1/8.
\]
To show this, we use Claim \ref{claim:theta_i-t_i}, which holds by assumption (\ref{eq:T_C_3}). We have that
\[
\begin{split}
\sum_{i=2}^T \left (\theta^{(i)} \right )^{1/6}
& =\sum_{i=2}^T e^{1/12(t_i+t_{i-1}-R)}\stackrel{t_{i-1} < R/2}{\leq} \sum_{i=2}^T e^{1/12(t_i-R/2)}\\
& \stackrel{t_i < \alpha^i t_0}{\leq} \sum_{i=1}^T e^{1/12\left (\alpha^i-1\right )R/2}= e^{-R/24 }\sum_{i=1}^T e^{ \alpha^i R/24}\\
& \leq e^{-R/24 } T e^{\alpha R/24}.
\end{split}
\]
Since $\alpha <1$ and $T=O\left( \ln R \right)$ by Lemma \ref{lemma:T=O(lnR)}, the bound follows.

\section{Proof of (\ref{eq:lower_bd_probability})} \label{sect:linearly_many_vertices}
We conclude our proof showing that the sum of the the error terms obtained in (\ref{eq:P_azuma}) and (\ref{eq:def_M2}) is of order 
$o(1)$, i.e., that (\ref{eq:lower_bd_probability}) is $1-o(1)$. 

Since the event $\mathcal{E}_i$ (defined in Equation (\ref{eq:N})) is realized, the lower bound on $N_i$ given by 
(\ref{eq:N_iFinalBounds}) implies that there is a constant $\xi^{(1)}>0$ such that for $N$ large enough
\[
M_i^{(1)}(N)\geq \xi^{(1)} N^{1-\alpha}, \quad 0\leq i<T.
\]
A straightforward calculation gives
\[
\sum_{i=1}^{T-1} e^{-M_i^{(1)}(N)}\leq T e^{-\xi^{(1)}N^{1-\alpha}}=o(1),
\]
where we used the fact that $T=O(\ln R )$.

Now we seek the analogous relation for $M_i^{(2)}(N)$, which was defined in (\ref{eq:def_M2}).
Using the definition of $\theta^{(i)}$ and the fact that $t_i\leq R/2$ for all $i\geq 0$, we see that there is a constant $\xi>0$ such that for all $1\leq i < T$
\[
\begin{split}
t_i (\theta^{(i)})^{2/3}
& \lesssim \xi Re^{(t_i-R/2)/3} \stackrel{
Claim~\ref{claim:theta_i-t_i}}{\leq}  \xi Re^{(\alpha^i R/2-R/2)/3} \\
& 
\leq \xi Re^{(\alpha -1)R/6}
\lesssim e^{(\alpha -1)R/12},
\end{split}
\]
where the last asymptotic inequality holds for large enough $N$, and the one before the last is due to the fact that $\alpha<1$.

Therefore, there is a constant $\xi^{(2)}>0$ such that for $N$ sufficiently large for all $0< i<T$ 
\[
M_i^{(2)}(N)\stackrel{(\ref{eq:def_M2})}{\geq} \xi^{(2)} N^{(1-\alpha )/6}.
\]
which implies
\[
\sum_{i=1}^{T-1} e^{-M_i^{(2)}(N)}\leq T e^{-\xi^{(2)} N^{(1-\alpha)/6}}=o(1).
\]
Hence (\ref{eq:sum_N_i'}) holds with probability bounded from below by
\[
\prod_{i=1}^{T-1}\left (1-e^{-M_i^{(1)}(N)}\right )\left (1-e^{-M_i^{(2)}(N)}\right )\geq  1-\sum_{i=1}^{T-1}\left (e^{-M_i^{(1)}(N)}+e^{-M_i^{(2)}(N)}\right ) = 1-o(1).
\]

\nocite{}
\bibliographystyle{amsalpha}
\bibliography{bibliography_1}

\newcommand{\etalchar}[1]{$^{#1}$}
\providecommand{\bysame}{\leavevmode\hbox to3em{\hrulefill}\thinspace}
\providecommand{\MR}{\relax\ifhmode\unskip\space\fi MR }
\providecommand{\MRhref}[2]{%
  \href{http://www.ams.org/mathscinet-getitem?mr=#1}{#2}
}
\providecommand{\href}[2]{#2}
\begin{thebibliography}{KPK{\etalchar{+}}10}

\bibitem[AB02]{ar:StatMechs}
R.~Albert and A.-L. Barab\'asi, \emph{Statistical mechanics of complex
  networks}, Reviews of Modern Physics \textbf{74} (2002), 47--97.

\bibitem[ACM14]{ACM11}
H.~Amini, R.~Cont, and A.~Minca, \emph{Resilience to contagion in financial
  networks}, Mathematical Finance (2014), 37 pages.

\bibitem[AF14]{AmFou2014}
H.~Amini and N.~Fountoulakis, \emph{Bootstrap percolation in power-law random
  graphs}, Journal of Statistical Physics \textbf{155} (2014), 72--92.

\bibitem[AL03]{AdL03}
J.~Adler and U.~Lev, \emph{Bootstrap percolation: visualizations and
  applications}, Brazilian Journal of Physics \textbf{33} (2003), no.~3,
  641--644.

\bibitem[Ami10]{Am-nn}
H.~Amini, \emph{Bootstrap percolation in living neural networks}, Journal of
  Statistical Physics \textbf{141} (2010), 459--475.

\bibitem[BA99]{ar:BarAlb}
A.-L. Barab\'asi and R.~Albert, \emph{Emergence of scaling in random networks},
  Science \textbf{286} (1999), 509--512.

\bibitem[BFM15]{bode_fountoulakis_mueller}
M.~Bode, N.~Fountoulakis, and T.~M{\"u}ller, \emph{On the largest component of
  a hyperbolic model of complex networks"}, Electronic Journal of
  Combinatorics, to appear, 43 pages, 2015.

\bibitem[BJR07]{BJR}
B.~Bollob\'as, S.~Janson, and O.~Riordan, \emph{The phase transition in
  inhomogeneous random graphs}, Random Structures and Algorithms \textbf{31}
  (2007), 3--122.

\bibitem[CF15]{candellero_fountoulakis}
E.~Candellero and N.~Fountoulakis, \emph{Clustering and the hyperbolic geometry
  of complex networks}, Internet Mathematics, to appear, 51 pages, 2015.

\bibitem[CL02a]{ChungLu1+}
F.~Chung and L.~Lu, \emph{The average distances in random graphs with given
  expected degrees}, Proc. Natl. Acad. Sci. USA \textbf{99} (2002),
  15879--15882.

\bibitem[CL02b]{ChungLuComp+}
\bysame, \emph{Connected components in random graphs with given expected degree
  sequences}, Annals of Combinatorics \textbf{6} (2002), 125--145.

\bibitem[CL06]{ChungLu2006}
\bysame, \emph{Complex graphs and networks}, CBMS Regional Conference Series in
  Mathematics, vol. 107, Published for the Conference Board of the Mathematical
  Sciences, Washington, DC; by the American Mathematical Society, Providence,
  RI, 2006.

\bibitem[CLR79]{ChLeRe:79}
J.~Chalupa, P.~L. Leath, and G.~R. Reich, \emph{Bootstrap percolation on a
  {B}ethe lattice}, Journal of Physics C: Solid State Physics \textbf{12}
  (1979), L31--L35.

\bibitem[Coo04]{Cooper2002}
C.~Cooper, \emph{The cores of random hypergraphs with a given degree sequence},
  Random Structures Algorithms \textbf{25} (2004), no.~4, 353--375.

\bibitem[CW06]{Wormald}
J.~Cain and N.~Wormald, \emph{Encores on cores}, Electron. J. Combin.
  \textbf{13} (2006), no.~1, Research Paper 81, 13 pp. (electronic).

\bibitem[GPP12]{ar:Kosta}
L.~Gugelmann, K.~Panagiotou, and U.~Peter, \emph{Random hyperbolic graphs:
  degree sequence and clustering}, Proceedings of the 39th International
  Colloquium on Automata, Languages and Programming (A.~ Czumaj et al. Eds.),
  Lecture Notes in Computer Science, vol. 7392, 2012, pp.~573--585.

\bibitem[J{\L}TV12]{janson_luczak_turova_vallier}
S.~Janson, T.~{\L}uczak, T.~Turova, and T.~Vallier, \emph{Bootstrap percolation
  on the random graph {$G_{n,p}$}}, Ann. Appl. Probab. \textbf{22} (2012),
  no.~5, 1989--2047.

\bibitem[KM15]{ar:KiMits}
M.~Kiwi and D.~Mitsche, \emph{A bound on the diameter of random hyperbolic
  graphs}, Proceedings of the 12th Workshop on Analytic Algorithmics and
  Combinatorics, ANALCO (R. Sedgewick and R.D. Ward, Eds.), SIAM, 2015,
  pp.~26--39.

\bibitem[KPK{\etalchar{+}}10]{ar:Krioukov}
D.~Krioukov, F.~Papadopoulos, M.~Kitsak, A.~Vahdat, and M.~Bogu{\~n}{\'a},
  \emph{Hyperbolic geometry of complex networks}, Phys. Rev. E (3) \textbf{82}
  (2010), no.~3, 036106.

\bibitem[PSW96]{PitSpWorm}
B.~Pittel, J.~Spencer, and N.~Wormald, \emph{Sudden emergence of a giant
  {$k$}-core in a random graph}, J. Combin. Theory Ser. B \textbf{67} (1996),
  no.~1, 111--151.

\bibitem[SDS02]{sadhsh02}
S.~Sabhapandit, D.~Dhar, and P.~Shukla, \emph{Hysteresis in the random-field
  {I}sing model and bootstrap percolation}, Physical Review Letters \textbf{88}
  (2002), no.~19, 197202.

\bibitem[S{\"o}d02]{ar:s02}
B.~S{\"o}derberg, \emph{General formalism for inhomogeneous random graphs},
  Phys. Rev. E \textbf{66} (2002), 066121.

\bibitem[TBF06]{tobifi06}
C.~Toninelli, G.~Biroli, and D.~S. Fisher, \emph{Jamming percolation and glass
  transitions in lattice models}, Physical Review Letters \textbf{96} (2006),
  no.~3, 035702.

\bibitem[TE09]{ET09}
T.~Tlusty and J.P. Eckmann, \emph{Remarks on bootstrap percolation in metric
  networks}, Journal of Physics A: Mathematical and Theoretical \textbf{42}
  (2009), 205004.

\bibitem[vdH]{bk:vdH}
R.~van~der Hofstad, \emph{Random graphs and complex networks}, available at
  \texttt{http://www.win.tue.nl/$\sim$rhofstad/}.

\bibitem[WS98]{ar:WatStrog98}
D.~J. Watts and S.~H. Strogatz, \emph{Collective dynamics of ``small-world"
  networks}, Nature \textbf{393} (1998), 440--442.

\end{thebibliography}

\end{document}